\documentclass{amsart}
\usepackage[utf8]{inputenc}

\usepackage{mathptmx}
\usepackage[margin=1.25in]{geometry}
\usepackage{graphics}
\usepackage{thmtools}

\usepackage{amsthm}
\usepackage{amsbsy,amsmath,amssymb,amscd,amsfonts}
\usepackage{xstring}
\usepackage[pagebackref=true]{hyperref}
\usepackage[nameinlink,capitalize,noabbrev]{cleveref}

\DeclareRobustCommand{\hrefs}[1]{%
  \def\temphrefsurl{#1}
  \IfBeginWith{\temphrefsurl}{http://}{%
    \StrGobbleLeft{\temphrefsurl}{7}[\temphrefsdisplayurl]
  }{%
    \IfBeginWith{\temphrefsurl}{https://}{%
      \StrGobbleLeft{\temphrefsurl}{8}[\temphrefsdisplayurl]
    }{%
      \def\temphrefsdisplayurl{\temphrefsurl}
    }%
  }%
  \href{\temphrefsurl}{\texttt{\temphrefsdisplayurl}}
}

\usepackage{graphicx,float,latexsym,color}
\usepackage[export]{adjustbox}

\usepackage[font={scriptsize,it},skip=5pt]{caption}
\usepackage{subcaption}

\usepackage{makecell}

\usepackage[dvipsnames]{xcolor}

\newtheorem{theorem}{Theorem}
\newtheorem*{theorem*}{Theorem}
\newtheorem{observation}{Observation}
\newtheorem{proposition}{Proposition}
\newtheorem{conjecture}{Conjecture}
\newtheorem{corollary}{Corollary}
\newtheorem{lemma}{Lemma}
\theoremstyle{remark}

\theoremstyle{definition}
\newtheorem{definition}{Definition}

\hypersetup{
    pdftoolbar=true,        
    pdfmenubar=true,        
    pdffitwindow=false,     
    pdfstartview={FitH},    
    colorlinks=true,       
    linkcolor=OliveGreen,          
    citecolor=blue,        
    filecolor=black,      
    urlcolor=red           
}

\usepackage{listings}
\usepackage{lineno}

\newcommand{\tb}[1]{\textbf{#1}}

\newcommand{\ti}[1]{\textit{#1}}

\usepackage{comment}
\usepackage{microtype}
\usepackage{footnote}

\newcommand{\A}{\mathcal{A}}

\newcommand{\E}{\mathcal{E}}
\newcommand{\K}{\mathcal{K}}
\newcommand{\T}{\mathcal{T}}

\newcommand{\F}{\mathcal{F}}
\renewcommand{\L}{\mathcal{L}}

\renewcommand{\T}{\mathcal{T}}

\newcommand{\Tt}{\T_{\triangle}}
\newcommand{\Et}{\E_{\triangle}}
\newcommand{\It}{I_{\triangle}}
\newcommand{\To}{\T_{\circ}}

\newcommand{\ol}{\overline}
\renewcommand{\l}{\lambda}

\newcommand{\torp}[2]{\texorpdfstring{#1}{#2}}

\newcommand{\rc}{\raisebox{0.3ex}{,}}
\newcommand{\rd}{\raisebox{0.3ex}{.}}

\title[Harmonious Poncelet loci]{Harmonious loci of Poncelet triangles about the incircle and their degeneracies}
\author[R. Garcia]{Ronaldo A. Garcia}
\author[M. Helman]{Mark Helman}
\author[D. Reznik]{Dan Reznik} 

\begin{document}

\begin{abstract}
We tour several Euclidean properties of Poncelet triangles inscribed in an ellipse and circumscribing the incircle, including loci of triangle centers and envelopes of key objects. We also show that a number of degenerate behaviors are triggered by the presence of an equilateral triangle in the family. 
\end{abstract}

\maketitle

\section{Introduction}
\label{sec:intro}
Poncelet's porism is a 1d family of $n$-gons with vertices on a first conic $\E$ and with sides tangent to a second conic $\E_c$ (also called the `caustic'). For such a porism to exist, $\E$ and $\E_c$ must be positioned in $\mathbb{R}^2$ so as to satisfy `Cayley's condition' \cite{dragovic11}. While the porism is projectively invariant — a conic pair is the projective image of two circles where the `Poncelet map' is linearized, — we have found that Poncelet porisms of triangles ($n=3$) are a wellspring of interesting Euclidean phenomena involving the dynamical geometry of classical objects associated with the triangle (centers, circles, lines/axes, etc.).

Two canonical cases are shown in \cref{fig:bicentric-confocal}: on the left, Chapple's porism, for which $\E,\E_c$ are circles; on the right the \ti{elliptic billiard}, where $\E,\E_c$ are confocal (right). In the former case, the loci of both the incenter $X_1$ and the circumcenter $X_3$ are stationary by definition (the incircle and circumcircle are fixed), the $X_k$ notation is after Kimberling \cite{etc}. The barycenter $X_2$ and orthocenter $X_4$ trace circles. In the confocal case, the perimeter of triangles is known to be conserved \cite{sergei91}, and the loci of all aforementioned centers sweep ellipses \cite{fierobe2021-circumcenter,garcia2020-ellipses,olga14}. Furthermore, the confocal family conserves the ratio of inradius-to-circumradius (equivalently, the sum of internal angle cosines) \cite{garcia2020-new-properties}; the cosine sum has been generalized to all billiard $n$-gons \cite{akopyan2020-invariants}. Significantly, none of foci, perimeters, angles, radii, triangle centers' loci, much less their foci, are projectively equinvariant.

\begin{figure}
\centering
\includegraphics[width=\linewidth]{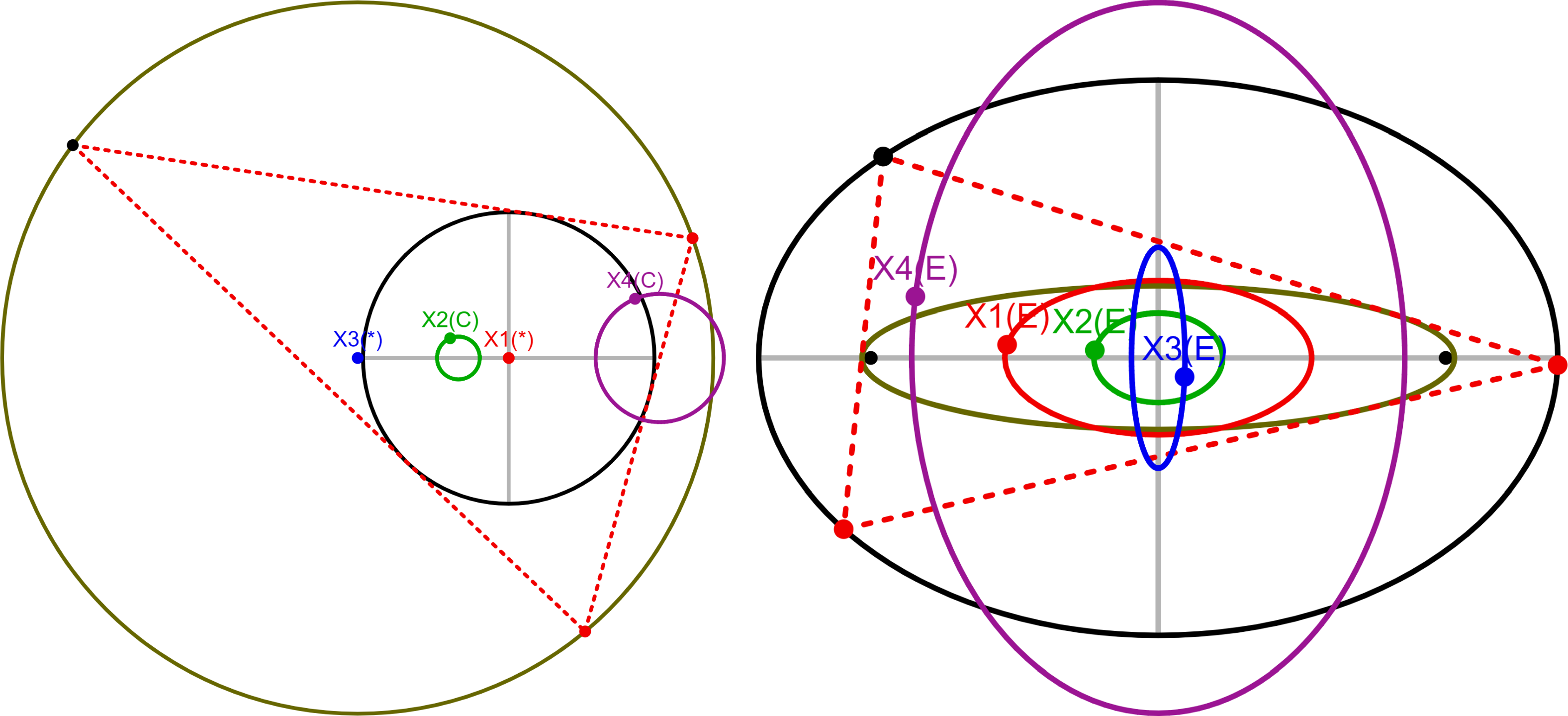}
\caption{\tb{left}: the loci of the incenter $X_1$, the barycenter $X_2$, the circumcenter $X_3$, and the orthocenter $X_4$ (we use Kimberling's notation \cite{etc}) over Chapple's porism, a 1d family of triangles inscribed in a circle and circumscribing another one (studied in \cite{odehnal2011-poristic}). Note that an (*), (C), (E), (L) next to $X_k$ flags the locus as stationary, circular, elliptic, or line-like, respectively. Live: \hrefs{https://bit.ly/4kMHGql}; \tb{right}: elliptic loci of said triangle centers over the confocal family, studied in \cite{reznik2020-intelligencer}. Live: \hrefs{https://bit.ly/4lwCFn0}}
\label{fig:bicentric-confocal}
\end{figure}

\begin{figure}
\noindent
\includegraphics[width=\linewidth]{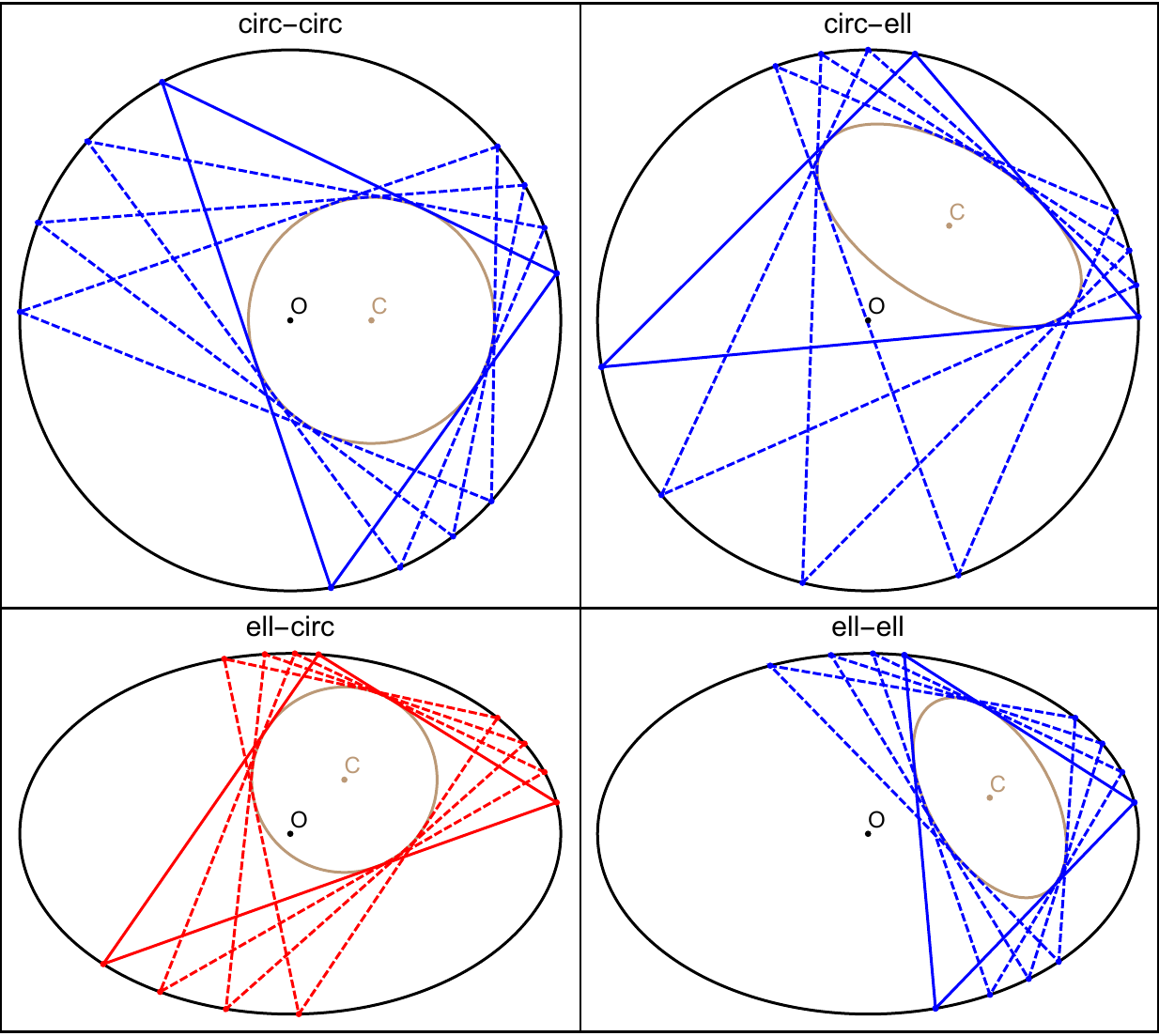}
\caption{\tb{upper left}: outer and inner Poncelet conics $\E,\E_c$ are circles (Chapple's porism); \tb{upper right}: outer $\E$ is a circle; \tb{bottom left}: inner $\E_c$ is a circle (in red since focus of present work); \tb{bottom right}: neither $\E$ nor $\E_c$ is a circle. Video: \hrefs{https://youtu.be/UTdGwAIjuT8}}
\label{fig:four-porisms}
\end{figure}

Guided by experiment, we have found that certain $\E,\E_c$ pairs form equivalent classes of Euclidean phenomena. One subdivision labels such pairs as concentric and/or axis-aligned (e.g., the confocal case is both). Relevant for the present study is to classify triangle porisms on whether either Poncelet conic is a circle or not. Referring to \cref{fig:four-porisms}, we obtain the following four combinations:

\begin{enumerate}
\item Both conics are circles: this is Chapple's porism \cite{centina15}. Loci of its \ti{triangle centers}, e.g., incenter $X_1$, barycenter $X_2$, circumcenter $X_3$, orthocenter $X_4$ have been studied in \cite{georgiev2012-poncelet,alexey14,odehnal2011-poristic,reznik2020-ballet}, see \cref{fig:bicentric-confocal} (left). The $X_k$ notation — also $X(k)$ — is after Kimberling \cite{etc}. For easy reference, the names of all $X_k$ studied appear in \cref{app:centers}.
\item Only $\E$ is a circle. This system, always an affine image of (4), lends itself to a symmetric parametrization which is the main engine for our analysis.
\item Only $\E_c$ is a circle: here we will tour its manifold Euclidean properties.
\item Neither is circular: Euclidean phenomena are sparser, some are still discernible, e.g., the locus of certain centers are always conics, the locus of barycenter $X_2$ and orthocenter $X_4$ always homothetic to $\E$ (see below), etc.
\end{enumerate}

For case (3), we start by identifying triangle centers whose loci are conics. In \cite{helman2021-theory} we show that if a triangle center is a fixed linear combination of the barycenter $X_2$ and the circumcenter $X_3$, then its locus will be a conic. Nevertheless, other centers may still sweep conics for some choices of conic pairs. Still lacking a predictive theory for the latter, we typically use numerical methods to pinpoint their loci as conic, as in \cite{garcia2020-ellipses}.

Below we describe many curious dependencies between the shape and foci of certain loci and the geometry of either Poncelet conic. A surprising find has been that the presence of a single equilateral triangle in the family triggers a host of interesting locus degeneracies.

\subsection{Article structure}

We begin with \cref{sec:prelims}, describing two algebraic tool-kits used throughout the paper:

\begin{itemize}
\item \cref{sec:cayley}: we derive the `Cayley' condition for closure of a triangle family interscribed between an ellipse $\E$ and a circle $\K$ (center $C$), and provide an explicit expression for the caustic radius as a function of the ellipse data. We then visualize the quartic $C$-isocurves of the radius of $K$.
\item \cref{sec:symmetric}: we present a parametrization for Poncelet triangles inscribed in the unit circle in $\mathbb{C}$ which keeps the intrinsic symmetry of the vertices, using the work on `Blaschke products' from \cite{daepp2019-ellipses}. We then present the generalization to Poncelet triangles interscribed between any two ellipses as described in \cite{helman2021-theory}, and then specialize it to the case where the inner ellipse is a circle, providing explicit formulas.
\end{itemize}

These are our main results:

\begin{itemize}
\item \cref{sec:loci}: we characterize remarkable properties of loci of certain triangle centers, including the barycenter, circumcenter, orthocenter, Euler center, and the Gergonne and Nagel points, denoted $X_2$, $X_3$, $X_4$, $X_5$, $X_7$, and Nagel's point $X_8$, respectively, after \cite{etc}. Their `behaviors', described in detail in \cref{sec:behaviors}, include:
\begin{itemize}
\item Axis-alignment, concentricity, and homothety with respect to $\E$.
\item The `railing' of the foci of certain centers to the axes of $\E$ or to the caustic center $C$.
\item Circularity: for no obvious reason, some centers always trace circles, regardless of $C$. For example, $X_{36}$, the inversive image of the incenter $X_1$ (stationary in our case) with respect to the (moving) circumcircle always sweeps a circle. We derive expressions for its center and radius.
\end{itemize}
Based on a previous result \cite[Conj.3]{helman2021-power-loci}, we conjecture that the loci of the Gergonne center $X_7$ and Nagel's point $X_8$ are conics if and only if (i) the caustic is a circle or (ii) the pair is confocal, \cref{sec:ellipticity}.
\item \cref{sec:degenerate}:
A series of curious, `degenerate' phenomena occurs if the Poncelet family contains an equilateral triangle. This is tantamount to $C$ lying on a special ellipse $\Et$, the locus of the centroids of all equilateral triangles inscribed in $\E$, derived in  \cite{stanev2019-equi}. The following degeneracies arise:
\begin{enumerate}
\item $C$ is a vertex of the locus of the circumcenter $X_3$.
\item The locus of the Euler center $X_5$ collapses to a segment.
\item Feuerbach's point $X_{11}$ becomes stationary on the incircle, and the reflection of the incenter $X_1$ about it ($X_{80}$) is a fixed point on $\E$.
\item The aspect ratio (major axis length over the minor one) of the elliptic loci of the circumcenter $X_3$ and the Gergonne point $X_7$ are invariant over all $C$ on $\E^{\triangle}$.
\item The locus of $X_{36}$ becomes a straight line (infinite radius circle).
\item The locus of $X_{59}$, the `isogonal conjugate' of Feuerbach's point $X_{11}$ (a type of inversive transformation \cite[Isogonal Conjugate]{mw}), normally a high-degree curve, collapses to an ellipse touching $\E$ at a special point.
\end{enumerate}
\end{itemize}

The following are provided for easy reference:

\begin{itemize}
\item \cref{app:affine-triples}: a table of `affine triples' used in symbolic calculations involving loci, reproduced from \cite[Table 1]{helman2021-theory}.
\item \cref{app:centers}: a table with detailed information about most triangle centers $X_k$ mentioned in the paper.
\item \cref{app:symbols}: a table with most symbols used in the article.
\end{itemize}



\subsection{Experiments and proofs}
An experimental discovery approach ``offers the possibility in some instances of formalizing the inductive leaps that the mathematical mind takes when confronted with what seems to be, logically speaking, incomplete evidence'' \cite[p.210]{davis1995-tri}. This has certainly been the case for us as most phenomena have been discovered via simulation, and are often proved with the aid of a Computer Algebra System (CAS), such as \ti{Maple} and/or \ti{Mathematica} \cite{maple2024,mathematica2024}. Upwards of some forty links to videos and animations (see our web app \cite{darlan2021-app}) are provided in figure captions and throughout the text. In animation screenshots, a triangle center is often displayed as $Xk(z)$ where $z\in{*,C,E,L}$ to indicate whether its locus is a point, circle, ellipse, or line, respectively, e.g., `$X11(C)$' indicates that Feuerbach's point $X_{11}$ sweeps a circle.

\section{Preliminaries}
\label{sec:prelims}
 
\subsection{Cayley closure condition}
\label{sec:cayley}
The Cayley condition specifies whether two ellipses $\{\E,\E_c\}$ admit a Poncelet family of $n$-gons \cite{dragovic11}. Referring to \cref{fig:tilted-caustic}, consider a Poncelet triangle family inscribed in $\E$ with center $O$ at the origin and semiaxis' lengths $a$ and $b$. Let $\E_c$ be nested within $\E$ and be centered at $O_c=[x_c,y_c]$ (sometimes called $C$), where $a_c,b_c$ denote its semiaxis'   lengths. Let $\theta$ be the angle of its major axis with respect to the major axis of $\E$. Note: when the caustic is a circle, $O_c$ is called $C$.

\begin{figure}
\centering
\includegraphics[width=.7\linewidth]{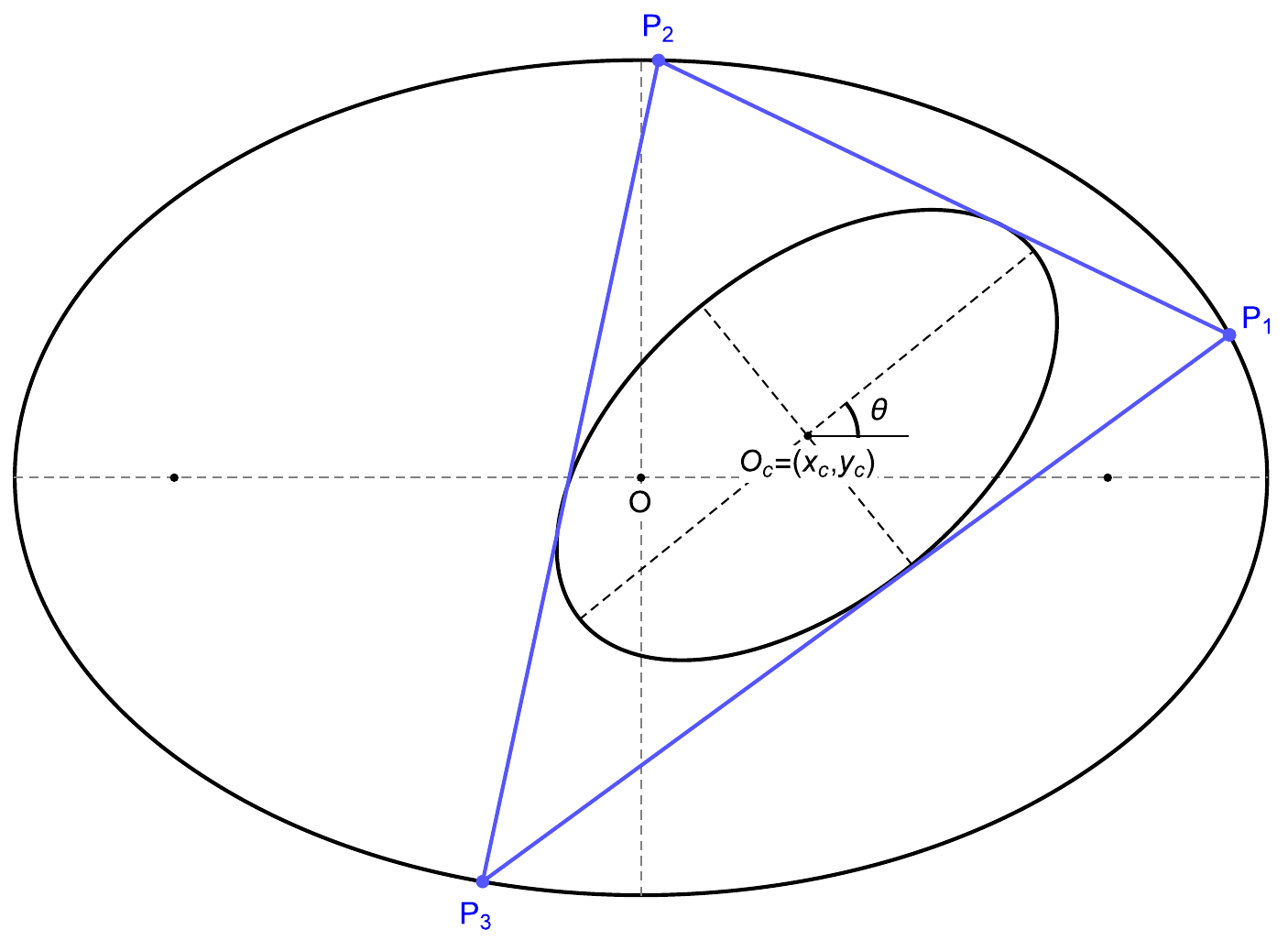}
\caption{A Poncelet triangle family inscribed in an ellipse centered at $O$ and circumscribing a caustic centered at $O_c$ and tilted by $\theta$ with respect to the former's major axis.}
\label{fig:tilted-caustic}
\end{figure}

In the calculations below, $c^2=a^2-b^2$, and $c_c^2=a_c^2-b_c^2$. Correcting our own expression in \cite[Section 6, eqn. (3)]{helman2021-power-loci}, originally obtained via CAS:

\begin{proposition}
The Cayley condition for the pair $(\E,\E_c)$ to admit a triangle family ($n=3$ polygons) is given by:
{\small
\begin{align}
&(a^4 b_c^4 + a_c^4 c^4+ b^4 b_c^4- 2 a^2 b^2 b_c^4  - 2c^4  a_c^2 b_c^2 ) \cos^4\theta - 8 a^2 b^2 c_c^2\, x_c y_c \sin\theta \cos\theta~+ \label{eqn:cayley}  \\
&\left[2 c_c^2 (a^2 + b^2) (a^2 y_c^2 - b^2 x_c^2) + 2 c_c^2 b^2 a^4 + 2( b^2 b_c^4-  a_c^4 c^2 -   b^4 c_c^2) a^2 + 2 b_c^2 (a_c^2 c^4 - b^4 b_c^2)\right]\cos^2\theta~+\nonumber \\
& a^4 y_c^4+ b^4 x_c^4 + 2 b^2 (a^2 a_c^2 - a^2 b^2 - b^2 b_c^2) x_c^2 + 2 a^2 b^2 x_c^2 y_c^2 - 2 a^2 (a^2 a_c^2 + a^2 b^2 - b^2 b_c^2)y_c^2~+ \nonumber \\
& (a a_c - a b - b b_c) (a a_c + a b - b b_c) (a a_c - a b + b b_c) (a a_c + a b + b b_c)=0\ldotp
\nonumber
\end{align}
}
\end{proposition}

Referring to \cref{fig:four-porisms} (bottom left), we specialize this to a circular caustic:

\begin{proposition}
\label{prop:circ-r}
The radius $r$ for a circular caustic (of Poncelet triangles) with center at $[x_c,y_c]$ is given by:
\[ r=\frac{b\sqrt{a^4-c^2 x_c^2}- a\sqrt{b^4+c^2 y_c^2}}{c^2}\cdot \]
\end{proposition}

\begin{proof}
When the caustic is the circle centered on $[x_c,y_c]$ with radius $r$, the condition in \cref{eqn:cayley} reduces to a biquadratic on $r$: 
\[ c^4 r^4 + 2 ( b^2 c^2 x_c^2 -a^2 c^2 y_c^2 - a^2 b^2 (a^2 + b^2)) r^2 + (a^2 b^2 - a^2 y_c^2 - b^2 x_c^2)^2=0\,. \]
Since the discriminant of the above equation is positive when $[x_c,y_c]$ is not a point of ellipse $\E$, it has has four real roots, or four complex roots.  For $[x_c,y_c]$ in the interior of $\E$ the above equation has four real roots, two being positive.
The smaller positive root of the above yields a caustic interior to $\E$. The other positive root $r^{+}>r$ is given by:
\[ r^{+}=\frac{b\sqrt{a^4-c^2 x_c^2}+ a\sqrt{b^4+c^2 y_c^2}}{c^2}\cdot \]
This yields a circle which completely encloses or intersects $\E$, potentially leading to a complex porism.
\end{proof}

As shown in \cref{fig:iso-radius}, the expression in \cref{prop:circ-r} prescribes for the interior of $\E$ a foliation of quartic isocurves of caustic radius for the center of the caustic.

\begin{figure}
\centering
\includegraphics[width=0.7\linewidth]{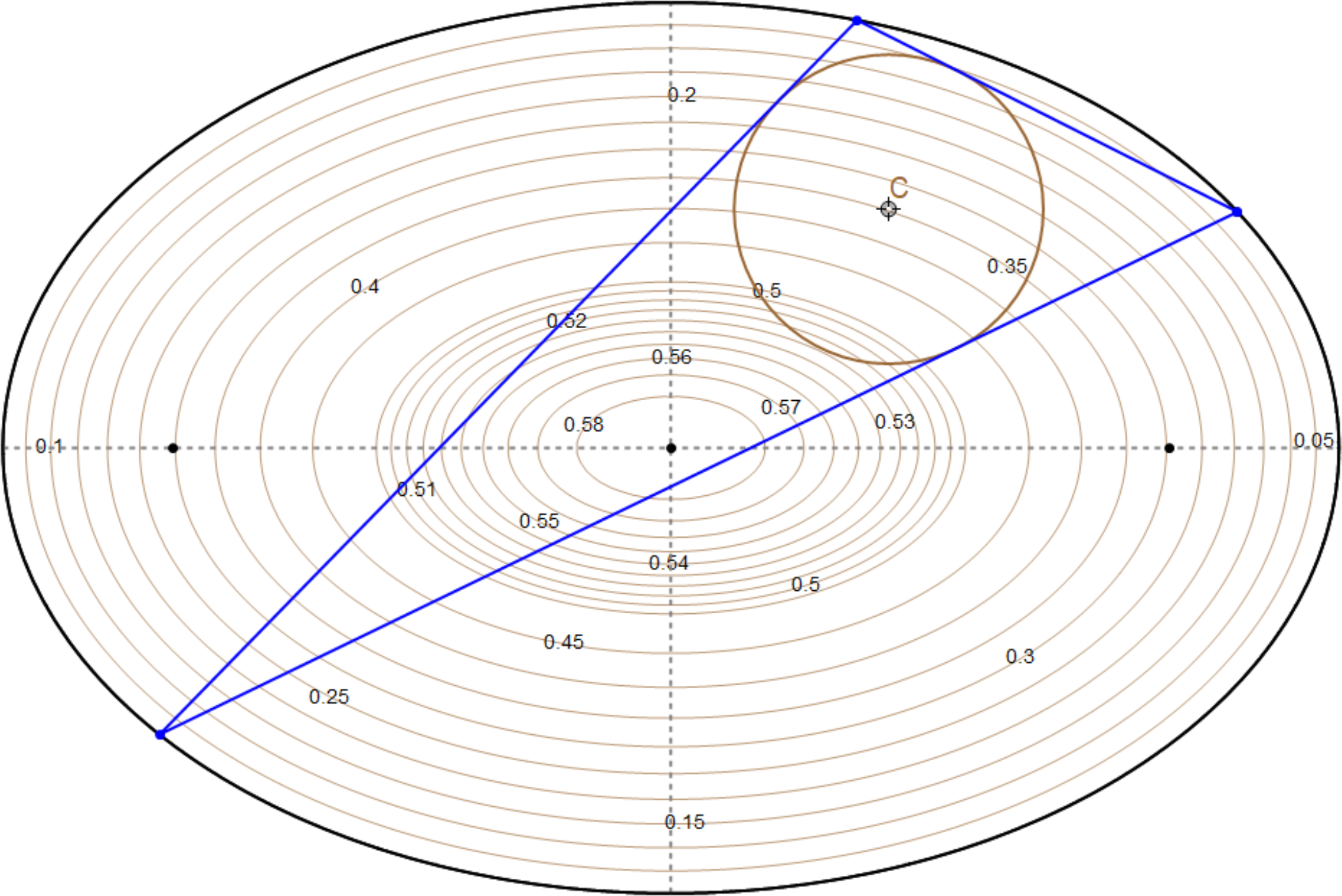}
\caption{The interior of $\E$ is foliated by quartic $C$-isocurves (brown). A particular caustic is shown for an outer ellipse with $a=1.5$, $b=1$, on the $r=0.35$ iso-curve. Also shown is a triangle (blue) in the family.}
\label{fig:iso-radius}
\end{figure}

\subsection{Symmetric parametrization}
\label{sec:symmetric}
Identifying $\mathbb{R}^2$ with $\mathbb{C}$, consider the following parameterization for Poncelet triangles inscribed in $\mathbb{T}$, the unit circle centered at the origin, as derived in \cite[Def. 3]{helman2021-power-loci} and based on the work in \cite{daepp2019-ellipses} on Blaschke products:
\begin{theorem}
For any Poncelet family of triangles inscribed in the unit circle $\mathbb{T}$ and circumscribing a nested ellipse with foci $f,g\in\mathbb{D}$ (the unit disk), parametrize its vertices $z_1,z_2,z_3\in\mathbb{T}$ as the following elementary symmetric polynomials:
\begin{align*}
z_1+z_2+z_3=& f+g+\l\ol f \ol g , \\
z_1 z_2+z_2 z_3+z_3 z_1=& f g+\l(\ol f+\ol g),\\
z_1 z_2 z_3=& \l,
\end{align*}
where the free parameter $\l=e^{i \theta}$, $\theta\in[0,2\pi]$.
\label{SymPar}
\end{theorem}

This is generalized to a Poncelet triangle family interscribed between any two nested ellipses $\E,\E_c$ by applying an affine transformation that sends $\mathbb{T}$ to $\E$. Let $z_1,z_2,z_3\in\E$ be the varying vertices of the Poncelet triangles. Then: 

\begin{theorem}
For any symmetric rational function $\F:\mathbb{C}^3\rightarrow\mathbb{C}$, the value of $\F(z_1,z_2,z_3)$ can be parameterized as a rational function of a parameter $\lambda$ on $\mathbb{T}$.
\end{theorem}
\begin{proof}
 By making a translation and a rotation as appropriate, we can assume that the outer ellipse $\E$ is centered at the origin and its major axis is aligned with the real axis of $\mathbb{C}$.

Let $a,b$ denote the semiaxis' lengths of $\E$, which satisfies $(x/a)^2+(y/b)^2=1$. Consider the affine transformation $\A(x,y)=(a x,b y)$ which sends the unit circle into $\E$. So $\A^{-1}(x,y)=(x/a,y/b)$. In the complex plane, $\A(z):=\frac{(a+b)}{2}z+\frac{(a-b)}{2}\overline{z}$. $\A^{-1}(z)=\frac{(1/a+1/b)}{2}z+\frac{(1/a-1/b)}{2}\overline{z}$.

Defining $\E_{pre}:=\A^{-1}(\E_c)$ and $z_i':=\A^{-1}(z_i)\in\mathbb{T}$ for $i\in{1,2,3}$, it can be shown that ${z_1',z_2',z_3'}$ form a Poncelet family of triangles interscribed between the unit circle $\mathbb{T}$ and the ellipse $\E_{pre}$. This enables us to parameterize the elementary symmetric polynomials in $z_1',z_2',z_3'$ using \cref{SymPar}.

Our objective is to parameterize the point $\F(z_1,z_2,z_3)=\F(\A(z_1'),\A(z_2'),\A(z_3'))$. Since $\F$ is symmetric on its $3$ inputs, the function $\F(\A(\cdot),\A(\cdot),\A(\cdot)):\mathbb{C}^3\rightarrow\mathbb{C}$ must also be symmetric on its $3$ inputs. Restricting $\A$ to $\mathbb{T}$, we can write $\A:\mathbb{T}\rightarrow\E$ as $\A(z):=\frac{(a+b)}{2}z+\frac{(a-b)}{2}\frac1z$, which is itself a rational function. Thus, $\F(z_1,z_2,z_3)=\F(\A(z_1'),\A(z_2'),\A(z_3'))$ is a symmetric rational function of $z_1',z_2',z_3'$. This means we can write $\F(z_1,z_2,z_3)$ as $\frac{\F_1(z_1',z_2',z_3')}{\F_2(z_1',z_2',z_3')}$, where $\F_1$ and $\F_2$ are symmetric polynomials. By the Fundamental Theorem of Symmetric Polynomials, we can express $\F_1(z_1',z_2',z_3')$ and $\F_2(z_1',z_2',z_3')$ as polynomials in $\sigma_1,\sigma_2,\sigma_3$ where $\sigma_1:=z_1'+z_2'+z_3'$, $\sigma_2:=z_1' z_2'+z_2' z_3'+z_3' z_1'$, and $\sigma_3:=z_1' z_2' z_3'$. Substituting the symmetric parameterization from \cref{SymPar}, we get that $\F(z_1,z_2,z_3)$ is a rational function of $\lambda\in\mathbb{T}$, as desired.
\end{proof}




Assume that the inner ellipse $\E_c$ is a circle with center $C=x_c+i y_c$, $x_c,y_c\in\mathbb{R}$, and radius $r$. Let $c^2=a^2-b^2$. Let, as above, $\E_{pre}:=\A^{-1}(\E_c)$. Then:


\begin{lemma}
$\E_{pre}$ is an axis-aligned ellipse with semi-major axis $r/b$ and semi-minor axis $r/a$, center $\A^{-1}(C)=x_c/a+i y_c/b$, and semi-focal length $r\frac{c}{a b}$, with foci given by $x_c/a+i (y_c/b\pm r\frac{c}{a b})$.
\label{CircleTransformLemma}
\end{lemma}
\begin{proof}
Since $\A^{-1}$ stretches the real and complex axis by factors of $1/a$ and $1/b$, respectively, $\E_{pre}:=\A^{-1}(\E_c)$ is an axis-aligned ellipse with semi-major axis and semi-minor axis equal to $r/b$ and $r/a$, respectively. In fact, it is a $90^{\circ}$-rotated copy of the outer ellipse $\E$ shrunken by a factor of $r/(a b)$.
Therefore its semi-focal length is given by $r\frac{c}{a b}$. Moreover, since ellipse centers are preserved by linear transformations, the center of $\E_{pre}$ is $\A^{-1}(C)=x_c/a+i y_c/b$. Since the major axis of $\E_{pre}$ is parallel to the imaginary axis of $\mathbb{C}$, its foci $f_{pre}$ and $g_{pre}$ are given by $x_c/a+i y_c/b\pm i r\frac{c}{a b}$, as desired.
\end{proof}

Under the assumptions of \cref{CircleTransformLemma}, we use the radius formula from \cref{prop:circ-r} to obtain the following:

\begin{corollary}
When the inner ellipse $\E_c$ of the Poncelet triangle family is a circle, the sum and product of the foci of $\E_{pre}$ as complex numbers are given by
\begin{align*}
f_{pre}+g_{pre}=&\frac{2x_c}{a}+\frac{2y_c}{b}\\
f_{pre} g_{pre}=& \frac{a^2+b^2}{c^2}+\frac{2i}{a b}\left(x_c y_c+\frac1c \sqrt{(a^4-c^2 x_c^2)(b^4-c^2 y_c^2)}\right)\cdot
\end{align*}
\end{corollary}

We can now calculate explicit formulas for any Poncelet family of triangles where the inner ellipse is a circle. In particular, we can calculate the loci of the barycenter $X_2$ and the circumcenter $X_3$, and then use the table in \cref{app:affine-triples} to explicitly calculate the loci of many triangle centers that can be conveniently written as a linear combination of the incenter $X_1$, the barycenter $X_2$, and the circumcenter $X_3$. An example for the locus of the incenter $X_1$ appears in \cite[ch. 7]{garcia2021-impa}. Many of our explicit results below are obtained by using a CAS to simplify such explicit loci expressions.

\section{Remarkable loci}
\label{sec:loci}
In \cite[Thm.2]{helman2021-power-loci} we proved that over the Poncelet triangle family, — henceforth shortened to `over Poncelet' — the locus of any triangle center which is a fixed linear combination of the barycenter $X_2$ and the circumcenter $X_3$ is an ellipse (a result restated in \cite[Thm.1]{helman2021-theory}). Here we focus on peculiar Euclidean properties of certain elliptic loci, such as their homothety and/or axis-alignment with respect to $\E$, and the special location of their center and/or foci. 

Below, let $\To$ (resp. $\T$) denote a family of Poncelet triangles circumscribing a circle (resp. interscribed between two generic conics $\E,\E_c$). In both cases, let the center of the outer conic be the origin and $C=[x_c,y_c]$ denote the caustic's center.

Since over $\To$ the incenter $X_1$ is fixed: 


\begin{corollary}
\label{cor:x1-ell}
Any triangle center $Y$ which is a fixed linear combination of the incenter $X_1$ and another center $X$ which sweeps a conic locus will also sweep a conic as its locus. 
\end{corollary}

Candidates for $X$ include (but aren't limited to) fixed linear combinations of the barycenter $X_2$ and the circumcenter $X_3$, since over Poncelet, these are guaranteed to sweep ellipses \cite[Thm.2]{helman2021-power-loci}. For such cases, \cref{cor:x1-ell} is identical to \cite[Corollary 1]{helman2021-theory}. An example for $Y$ is $X_{551}$ (resp. $X_{946}$), the midpoint of the incenter $X_1$ and the barycenter $X_2$ (resp. orthocenter $X_4$).


In the following sections, all proofs will follow the same methodology. (a) loci are expressed symbolically by plugging the barycentric coordinates of a given triangle center $X_k$, see \cref{tab:centers} (and the affine triples in \cref{tab:centers}), into the symmetric parametrization of a Poncelet triangle family, \cref{SymPar}. (b) structural elements — centers, foci, axes, etc., — are extracted upon simplification of (a) with a CAS.

\subsection{\torp{Loci of $X_k,~k=2,4,7$}{Loci of X(k), k=2,4,7}}

The barycenter $X_2$ (resp. orthocenter $X_4$) is the meet-point of lines from each vertex to the opposite side's midpoint (resp. altitude foot). The Gergonne point $X_7$ is the meet-point of lines from each vertex to the where the opposite side touches the inscribed circle \cite{etc}.

\begin{figure}
\centering
\includegraphics[width=.8\linewidth]{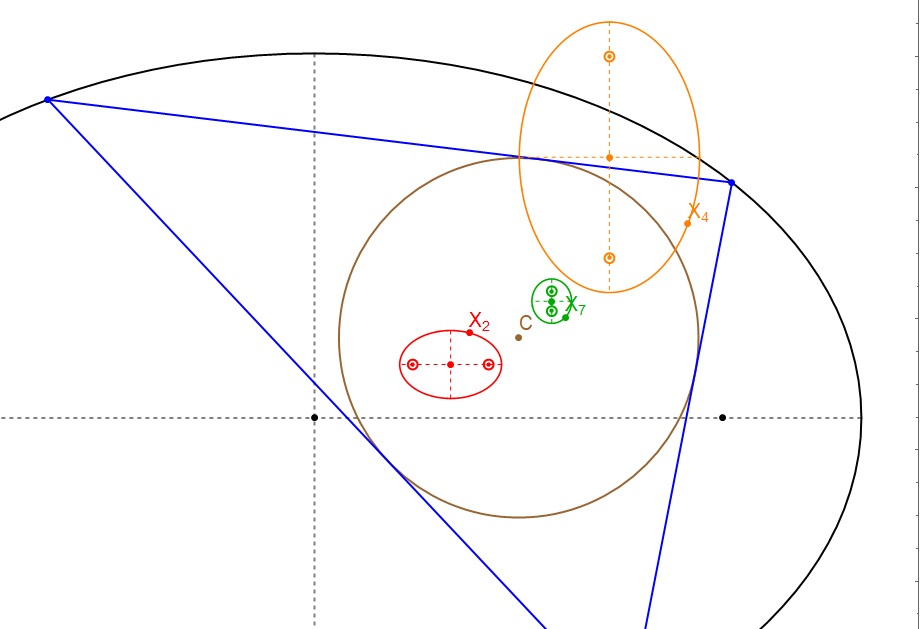}
\caption{Close-up of the first quadrant of $\E$, showing a Poncelet triangle (blue) and the circular caustic centered on $C=X_1$. Also shown are the loci of the barycenter $X_2$ (red), orthocenter $X_4$ (orange), and Gergonne point $X_7$ (green). Video: \hrefs{https://youtu.be/\_2OidqdtaaY}}
\label{fig:loci-x2x4x7}
\end{figure}

For next three propositions, refer to \cref{fig:loci-x2x4x7}. The following is a special case of \cite[Thm.1]{sergei2016-com}, which states that for any pair of conics admitting a Poncelet family, the loci of both vertex and area centroids sweep conics. For triangles, these coincide. 

Henceforth, let $\L_k$ denote the locus of a triangle center $X_k$ over $\To$, e.g., $\L_2$ is the locus of the barycenter $X_2$ over $\To$.

\begin{proposition}
\label{prop:locus-x2}
The locus $\L_2$ is homothetic to $\E$. Its center $C_2$ is collinear with the center of $\E$ and the caustic center $C$. The center $C_2$ and semi-axis lengths $a_2,b_2$ are given by: 
\begin{align*}
C_2&=\frac{2}{3}\left[x_c,y_c\right],\\
a_2^2&= {\frac {-4\,ab \left( {a}^{2}+{b}^{2} \right) \delta-4\,{b}^{4}{c}^{2}
x_c^{2}+4\,{a}^{4}{c}^{2} y_c^2+{a}^{2}{b}^{2} \left( {a}^{4}+6\,{a}^{2}{b}^{2}+{b}^{4} \right) }{9\,{b}^{2}{c}^{4}}}\rc\\
 b_2^2&=\frac{b^2}{a^2}a_2^2~\rc
\end{align*}
where $\delta=\sqrt{(b^4 +c^2 y_c^2) (a^4 - c^2 x_c^2)}$.
\end{proposition}

\begin{proof}
CAS simplification, e.g., as in \cite[Prop.7.1]{garcia2021-impa}. Note that $a_2/b_2=a/b$, as claimed.
\end{proof}

Referring to \cref{fig:gen-x4}, 
we defer the long proof to \cite[Thm.1]{garcia2025-x4-conjugate}:

\begin{theorem}
Over $\T$ the locus of the orthocenter $X_4$ is a conic homothetic to a $90^o$-rotated copy of $\E$. The locus center $C_4$ only depends on $a,b$ of $\E$ and $O_c=[x_c,y_c]$ and is given by:
\begin{align*} 
C_4&=\left[\frac{(a^2 + b^2)}{a^2}x_c\rc~\frac{(a^2 + b^2)}{b^2}y_c\right]\cdot
\end{align*}
\end{theorem}

Nevertheless, the semiaxes of the locus depend on the relative position of the two conics. If the caustic is a circle:

\begin{proposition}
The semiaxis lengths $a_4,b_4$ of $\L_4$ are given by:
\[a_4^2 = \frac{z_4}{a^2 b^4 c^4}\rc\quad b_4^2  = \frac{z_4}{b^2 a^4 c^4}\rc \]
where:
\begin{align*}
z_4 &=  (a^2+b^2)(a^8 y_c^2 + b^8 x_c^2-4 a^3 b^3 \delta) + a^8 b^4 +a^6 b^4 \left(6 b^2-z_4'\right)+a^4 b^6 \left(b^2-z_4'\right),\\
z_4' & = x_c^2+y_c^2~\ldotp
\end{align*}
\end{proposition}

\begin{proof}
Apply the generic expressions for the orthocenter locus' semiaxis given in \cite[Thm.1]{garcia2025-x4-conjugate}, using $a_c=b_c=r$, with $r$ obtained via \cref{eqn:cayley}. Notice that $a_4/b_4=a/b$, as required.
\end{proof}

\begin{figure}
\centering
\includegraphics[width=\linewidth]{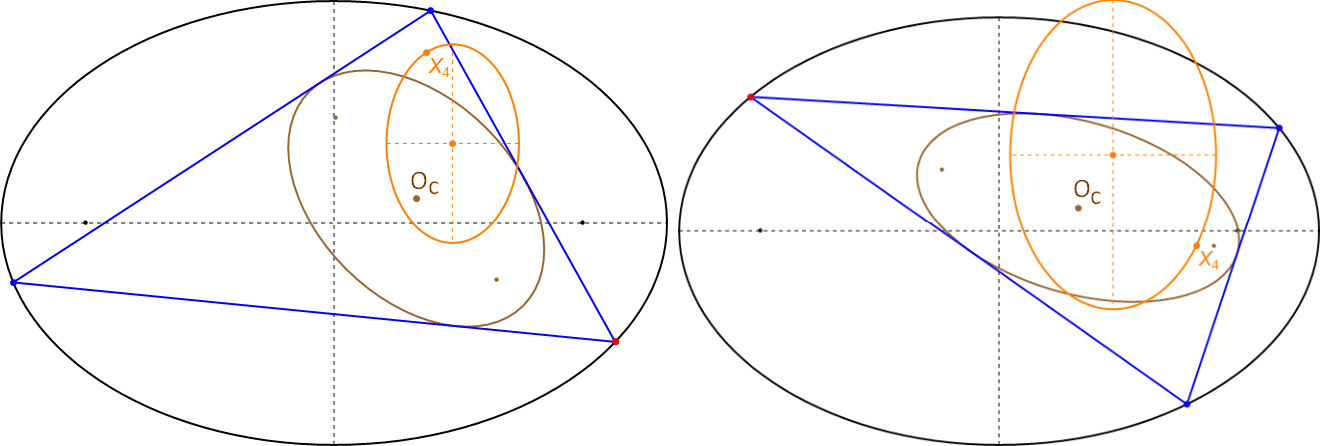}
\caption{\tb{left}: a Poncelet triangle (blue) interscribed between an outer ellipse $\E$ and a generic caustic $\E_c $, let $O_c$ be its center. The locus of the orthocenter $X_4$ (orange) is homothetic to a $90^o$-rotated copy of $\E$, see \cite{garcia2025-x4-conjugate}. \tb{right}: another caustic centered at the same $O_c$ which closes Poncelet. Notice that the center of the (new) locus remains at the same location. Video: \hrefs{https://youtu.be/tGZa3p6Q1BA}}
\label{fig:gen-x4}
\end{figure}


Regarding the Gergonne point $X_7$:

\begin{proposition}
\label{prop:x7-locus}
The locus $\L_7$ is an ellipse with major (resp. minor) axis is parallel to $\E$'s minor (resp. major axis). The coordinates of its center $C_7=[x_7,y_7]$ are given by:
\begin{align*}
x_7&=\frac{4 a x_c \left(4 a^5+a^3 b^2+a b^4-a \left(4 a^2-3 b^2\right) x_c^2-a b^2 y_c^2-3 b \delta \right)}{a^2\left(4 a^2-b^2\right)^2-\left(4 a^2-3 b^2\right)^2 x_c^2-a^2 b^2 y_c^2}\rc\\
y_7&=\frac{4 b y_c \left(4 b^5+a^2 b^3+a^4 b-b \left(4 b^2-3 a^2\right) y_c^2-a^2 b x_c^2-3 a \delta\right)}{b^2 \left(4 b^2-a^2\right)^2-\left(3 a^2-4 b^2\right)^2 y_c^2-a^2 b^2 x_c^2}\cdot
\end{align*}
Furthermore, the ratio of semi-axis lengths $a_7/b_7$ is given by: 
\[ \left(\frac{a_7}{b_7}\right)^2 = \frac{y_7\,x_c}{x_7\,y_c}~\cdot \]
\end{proposition}

We omit the rather long expressions for the semi-axis' lengths $a_7,b_7$ of $\L_7$.


\subsection{\torp{Loci of $X_k,~k=3,5,8$}{Loci of X(k), k=3,5,8}}

The circumcenter $X_3$ is the center of the circumscribed circle. The Euler center $X_5$ is the center of the nine-point circle, which passes through the sides' midpoints. Nagel's point $X_8$ is the meet-point of lines from each vertex to where the opposite touches an excircle \cite{etc}.

\begin{figure}
\centering
\includegraphics[width=.8\linewidth]{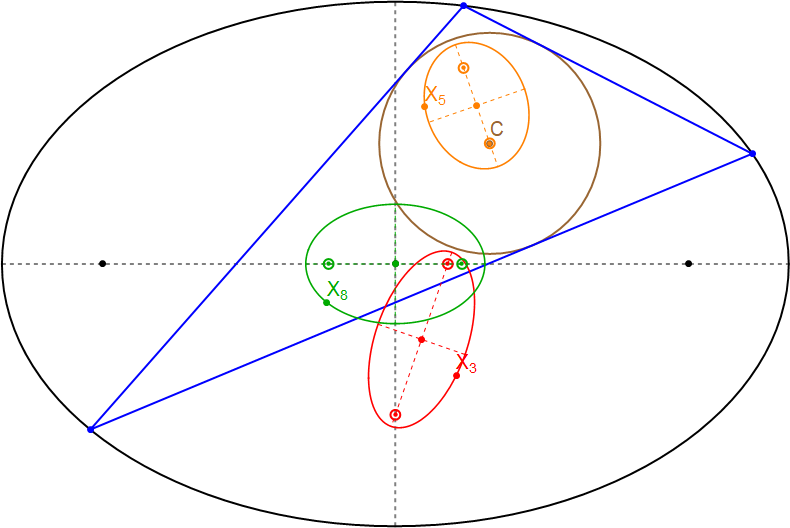}
\caption{A Poncelet triangle (blue) is shown inscribed in an outer ellipse and circumscribing the incircle centered at $C=X_1$. Also shown are the loci of the circumcenter $X_3$ (red), Euler center $X_5$ (orange), and Nagel's point $X_8$ (green). Video 1: \hrefs{https://youtu.be/EVd4vAWTlwg}; Video 2: \hrefs{https://youtu.be/n4Mh-LlQMLo}}
\label{fig:loci-x3x5x8}
\end{figure}

For the next three propositions, refer to \cref{fig:loci-x3x5x8}. We know that over any Poncelet family of triangles, the locus of the circumcenter $X_3$ is an ellipse \cite{helman2021-theory}. If the caustic is a circle. Interestingly:

\begin{proposition}
\label{prop:x3-foci}
The foci $F_3,F_3'$ of the locus $\L_3$ are collinear with $C$ and are each railed to an axis of $\E$, namely: 
\[ F_3=\left[x_c (1-(b/a)^2), 0\right],\,\,\,F_3'=\left[0,y_c(1-(a/b)^2)\right]\ldotp \]
Furthermore, the semi-axis lengths $a_3,b_3$ of the locus $\L_3$ are given by:
\begin{equation}
\label{eqn:a3b3}
a_3 = \delta_3(a/b)-\delta_3'(b/a),\quad b_3=\delta_3'-\delta_3,
\end{equation}
where $\delta_3=\frac{\sqrt{b^4+c^2 y_c^2}}{2b}$ and $\delta_3'=\frac{\sqrt{a^4-c^2 x_c^2}}{2a}$.
\end{proposition}

\begin{proof}
CAS-based manipulation and simplification of parametrization in \cref{sec:symmetric}.
\end{proof}

Note: we will use the same proof mechanics (symmetric parametrization + CAS simplification) below and said proofs will therefore be omitted.

\begin{observation}
In the special case when the circular caustic is concentric with $\E$, (i) $r=(ab)/(a+b)$, (ii) the invariant circumradius $R=(a+b)/2$, and (iii) the circumcenter $X_3$ sweeps a circle of radius $(a-b)/2$ concentric with $\E$ \cite[Section 3]{garcia2020-family-ties}. If the latter's center is thought of as two coinciding foci, note that each is still on an axis of $\E$, as required by \cref{prop:x3-foci}.
\end{observation}

Over any Poncelet family, the locus of Euler center $X_5$ is an ellipse because it is fixed linear combination of $X_2$ and $X_3$ \cite{helman2021-theory}, namely, it is the midpoint of $X_3$ and $X_4$ \cite{etc}. Interestingly:

\begin{proposition}
\label{prop:a5b5}
One focus $F_5$ of the locus $\L_5$ is at $C=[x_c,y_c]$ and the other $F'_5$ is given by:
\[ F_5'=\left[x_c \frac{1 + (b/a)^2}{2}\rc~y_c \frac{1 + (a/b)^2}{2}\right]\cdot \]
Furthermore, the semi-axis lengths $a_5,b_5$ of the locus $\L_5$ are given by:
\[a_5=\frac{\left(a^5+3 a^3 b^2\right)\delta_3-\left(b^5+3 a^2 b^3\right)\delta_3'}{a b c^2}\rc\quad b_5=\frac{\left(3 a^3+a b^2\right) \delta_3-\left(3 b^3+a^2 b\right)\delta_3'}{c^2}\rc\]
where $\delta_3,\delta_3'$ are as defined for the locus $\L_3$.
\end{proposition}

For a generic Poncelet family, the locus of Nagel's point $X_8$ is not a conic. Nevertheless, with a circular caustic:

\begin{proposition}
The locus $\L_8$ is an ellipse homothetic and concentric with $\E$. Its semi-axis lengths $a_{8},b_{8}$ are given by:
\[ a_8=\frac{\sqrt{\delta_8}}{b c^2}\rc\quad b_8=\frac{\sqrt{\delta_8}}{a c^2}\rc \]
where $\delta_8=(a^3 b + a b^3-2 \delta)^2+4 c^4 x_c^2 y_c^2$. Note that $a_8/b_8=a/b$ as claimed.
\end{proposition}

\subsection{\torp{Circular $X_{36}$}{Circular X(36)}}

The inversive image of the incenter $X_1$ with the respect to the circumcircle is called $X_{36}$ in \cite{etc}. Referring to \cref{fig:x36-gen}, the following is a surprising phenomenon (since it does not hold over $\T$):

\begin{proposition}
\label{prop:x36}
The locus $\L_{36}$ is a circle. Its center $C_{36}$ is given by:
{\small
\begin{align*}
C_{36}&=[x_{36}/{z_1},y_{36}/{z_1}],\\
x_{36}&=x_c \left({z_2}-6 a^4 b^4+6 a^4 b^2 y_c^2-3 a^2 b^6-2 a^2 b^4 x_c^2+3 a^2 b^4 y_c^2-8 a b^3 \,{\delta}+3 b^6 x_c^2\right),\\
y_{36}&=y_c \left(3 a^6 \left(y_c^2 - b^2\right)-a^4 \left(6 b^4+b^2 \left(2 y_c^2-3 x_c^2\right)\right)-8 a^3 b\,{\delta}+a^2 b^4 \left(b^2+6 x_c^2-y_c^2\right)-9 b^6 x_c^2\right),
\end{align*}}
\noindent where:
{\small
\begin{align*}
{z_1}&={z_2}-2 a^4 b^4-6 a^4 b^2 y_c^2+a^2 b^6-6 a^2 b^4 x_c^2-a^2 b^4 y_c^2-9 b^6 x_c^2,\\
{z_2}&=a^6 b^2-9 a^6 y_c^2-a^4 b^2 x_c^2~\ldotp
\end{align*}}
Furthermore, the radius $r_{36}$ of the locus $\L_{36}$ is given by:
\begin{align*}
r_{36} & =\frac{2(\delta_4 z_{36} -\delta_4' z'_{36})}{a^2 b^2 c^4-b^2(a^2+3 b^2)^2 x_c^2-a^2 (3 a^2+b^2)^2 y_c^2}\rc\\
z_{36} &= 5 b a^4 y_c^2-b^3 a^2 (a^2+y_c^2-x_c^2)+b^5 (a^2+3 x_c^2),\\
z'_{36} & = 5 a b^4 x_c^2-a^3 b^2 (b^2+x_c^2-y_c^2)+a^5 (b^2+3 y_c^2),
\end{align*}
 where $\delta_4=\sqrt{a^4-c^2 x_c^2}$ and $\delta_4'= \sqrt{b^4+c^2y_c^2}$.
\end{proposition}

\begin{figure}
\centering
\includegraphics[width=.8\linewidth]{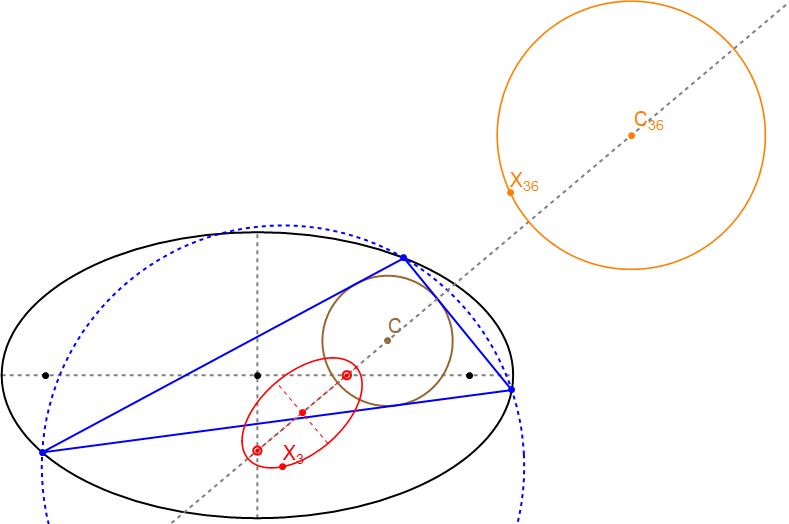}
\caption{The locus $\L_{36}$, the inversion of $C=X_1$ (fixed) with respect to the (moving) circumcircle (dashed blue), is a circle (orange). Also shown is the locus $\L_3$ (red ellipse), whose foci $F_3,F_3'$ lie on the semiaxes of the outer ellipse. Indeed, $F_3,F_3',C,C_{36}$ are collinear. Note that from the inversion operation, $X_3 X_1 X_{36}$ are dynamically collinear. Live: \hrefs{http://bit.ly/3GQbdSb}}
\label{fig:x36-gen}
\end{figure}

Referring to \cref{fig:x3x36-gen}, this follows directly from \cref{prop:x3-foci,prop:x36}:

\begin{corollary}
The major axis of the locus $\L_3$ passes through the center $C_{36}$ of the locus $\L_{36}$. Furthermore, the circumradius is minimized (resp. maximized) when the circumcenter $X_3$ is on a major vertex of its locus closest (resp. farthest) from $C=X_1$. At these configurations, the distance $|X_1 X_{36}|$ is maximized (resp. minimized), respectively.
\label{cor:x3x36}
\end{corollary}

\begin{figure}
\centering
\includegraphics[width=\linewidth]{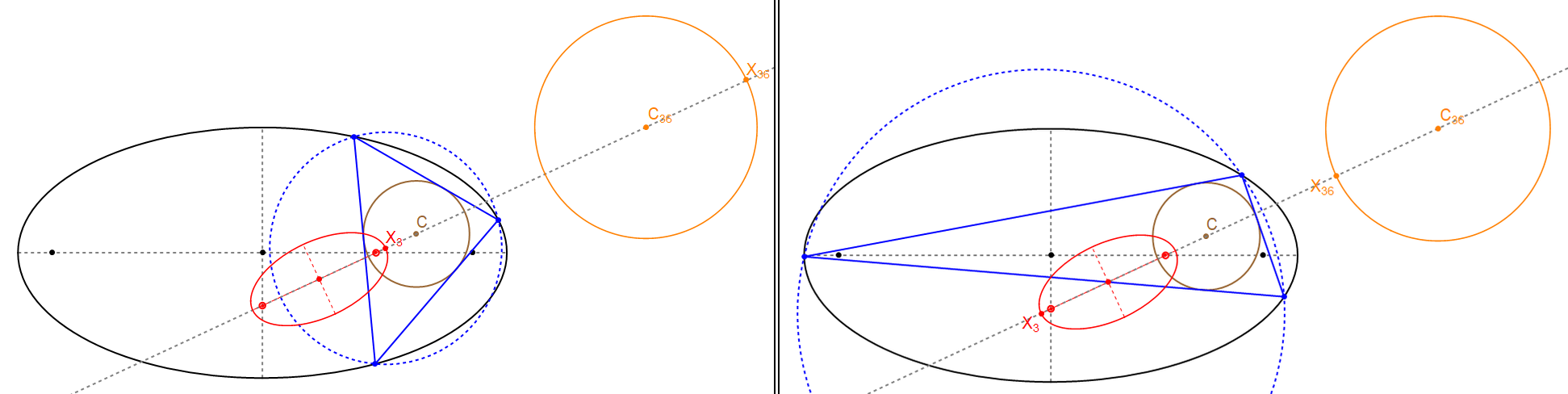}
\caption{The major axis of the locus $\L_3$ passes through the center $C_{36}$ of the circular $\L_{36}$. On the left (resp. right), the circumcenter $X_3$ is at the major vertex of the locus closest to (resp. farthest from) $C=X_1$, in which case the circumradius is minimized (resp. maximized) and $X_{36}$ is farthest from (resp. closest to) $C$. Live: \hrefs{https://bit.ly/4nRIltl}}
\label{fig:x3x36-gen}
\end{figure}

\begin{figure}
\centering
\includegraphics[width=.6\linewidth]{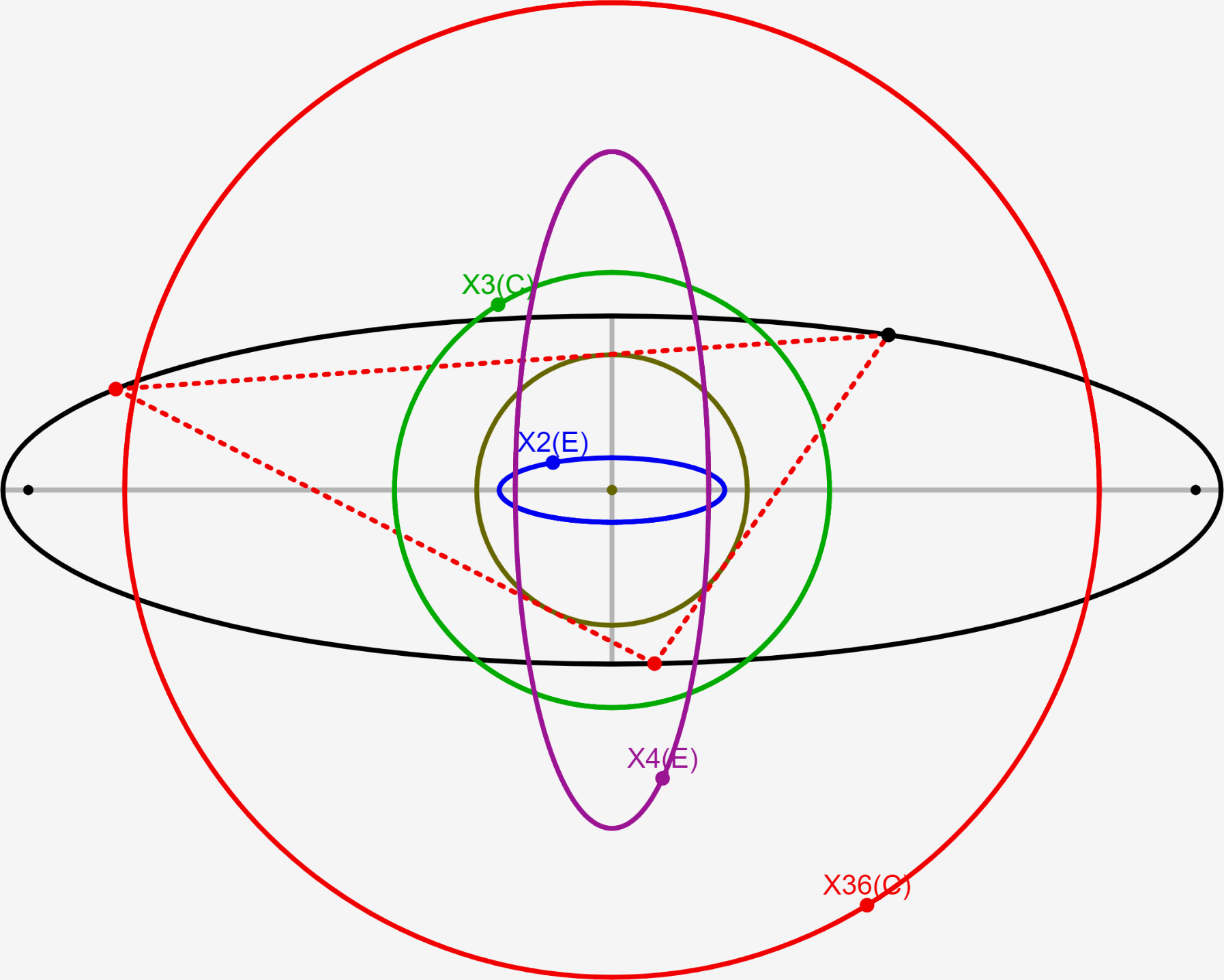}
\caption{With $C=X_1=[0,0]$, the locus $\L_{36}$ (red) is a circle concentric with $\E$. In the figure, $a=7/2$ and $b=1$, so $r_{36}=14/5$. Also shown are the loci of the barycenter $X_2$ (blue ellipse), the circumcenter $X_3$ (green circle) and the orthocenter $X_4$ (purple ellipse). Live: \hrefs{https://bit.ly/3GR1Wt1}}
\label{fig:x36-c0}
\end{figure}

Referring to \cref{fig:x36-c0}:

\begin{corollary}
Over $\To$, when $C$ is at the center of $\E$, the locus $\L_{36}$ is a circle concentric with $\E$, $r_{36}$ reduces to:
\[ r_{36} = \frac{2ab}{a-b}\cdot \]
\end{corollary}

\begin{proof}
Direct from setting $x_c=y_c=0$ in $r_{36}$ (\cref{prop:x36}).
\end{proof}

\subsection{\torp{Ellipticity of $X_k,~k=1,7,8$}{Ellipticity of X(k), k=1,7,8}}
\label{sec:ellipticity}

In \cite{garcia2020-ellipses,olga14} it is proved that the locus of the incenter $X_1$ is a conic if the Poncelet ellipse pair is confocal (elliptic billiard). In \cite[Conj.3]{helman2021-power-loci} it is conjectured that the incenter $X_1$ sweeps a conic if and only if the pair is confocal, and a proof is still outstanding.  

Experiments with ten different Poncelet families (see six of them in \cref{fig:x1x7x8-conic,fig:x1x7x8-nonconic}) are summarized in \cref{tab:loci-x1x7x8}. In turn, these suggest:

\begin{conjecture}
The loci $\L_7$ and/or $\L_8$ are conics if and only if (i) the Poncelet conic pair is confocal or (ii) the caustic is a circle (in which case the incenter $X_1$ is stationary).
\end{conjecture}

\begin{table}
{\small

\begin{tabular}{|r||c|c|c|c|c||l|}
\multicolumn{3}{c}{} & \multicolumn{3}{c}{\textbf{locus type for:}} & \multicolumn{1}{c}{} \\
\hline
\tb{Family} & \tb{Outer} & \tb{Inner} & 
{$X_1$} & {$X_7$} & {$X_8$} & \tb{Live} \\
\hline
Confocal & E & E & E & E & E & \hrefs{https://bit.ly/46KYk4D} \\
Chapple & C & C & P & C & C & \hrefs{https://bit.ly/4cqyGmN} \\
Incircle & E & C & P & E & E & \hrefs{https://bit.ly/3M50bHG} \\
Circ-caustic & E & C & P & E & E & \hrefs{https://bit.ly/4fK8Exy} \\
\hline
Homothetic & E & E & - & - & - & \hrefs{https://bit.ly/3M4sG8G} \\
Dual & E & E & - & - & - & \hrefs{https://bit.ly/3WHkQqg} \\
Conf. Excentrals & E & E & - & - & - & \hrefs{https://bit.ly/3WGARgk} \\
Inellipse & C & E & - & - & - & \hrefs{https://bit.ly/4fKZPUv} \\
Brocard & C & E & - & - & - & \hrefs{https://bit.ly/4fKQmMH} \\
MacBeath & C & E & - & - & - & \hrefs{https://bit.ly/3M3xtaj} \\
\hline
\end{tabular}

\begin{tabular}{|r||c|c|c|c|c||l|}
\hline
\tb{Family} & \tb{Outer} & \tb{Inner} & 
{$X_1$} & {$X_7$} & {$X_8$} & \tb{Live} \\
\hline
Confocal & E & E & E & E & E & \hrefs{https://bit.ly/46KYk4D} \\
Chapple & C & C & P & C & C & \hrefs{https://bit.ly/4cqyGmN} \\
Incircle & E & C & P & E & E & \hrefs{https://bit.ly/3M50bHG} \\
Circ-caustic & E & C & P & E & E & \hrefs{https://bit.ly/4fK8Exy} \\
\hline
Homothetic & E & E & - & - & - & \hrefs{https://bit.ly/3M4sG8G} \\
Dual & E & E & - & - & - & \hrefs{https://bit.ly/3WHkQqg} \\
Conf. Excentrals & E & E & - & - & - & \hrefs{https://bit.ly/3WGARgk} \\
Inellipse & C & E & - & - & - & \hrefs{https://bit.ly/4fKZPUv} \\
Brocard & C & E & - & - & - & \hrefs{https://bit.ly/4fKQmMH} \\
MacBeath & C & E & - & - & - & \hrefs{https://bit.ly/3M3xtaj} \\
\hline
\end{tabular}
}
\caption{The type of locus for $X_k,k=1,7,8$ for 10 Poncelet conic configurations. `Inner' and `Outer' indicate if such conics are either ellipses (E) or circles (C). Under columns labeled $X_k$ ($k=1,7,8$), symbols `E', `C', `P', `-' indicate if the corresponding locus is an ellipse, a circle, a point, or a non-conic, respectively. The pattern that emerges is that the Gergonne center $X_7$ and Nagel's point $X_8$ can only be conics if `Inner' is a circle, in which case the incenter $X_1$ is a point. The confocal family is one exception: the inner conic is a (confocal) ellipse, and all three centers sweep conics. Note: the last column gives a hyperlink to an animated view of the family.}
\label{tab:loci-x1x7x8}
\end{table}

Note that for the case of the bicentric pair, the loci $\L_7$ and $\L_8$ are both circles \cite{odehnal2011-poristic}.

\begin{figure}
\centering
\includegraphics[width=\linewidth]{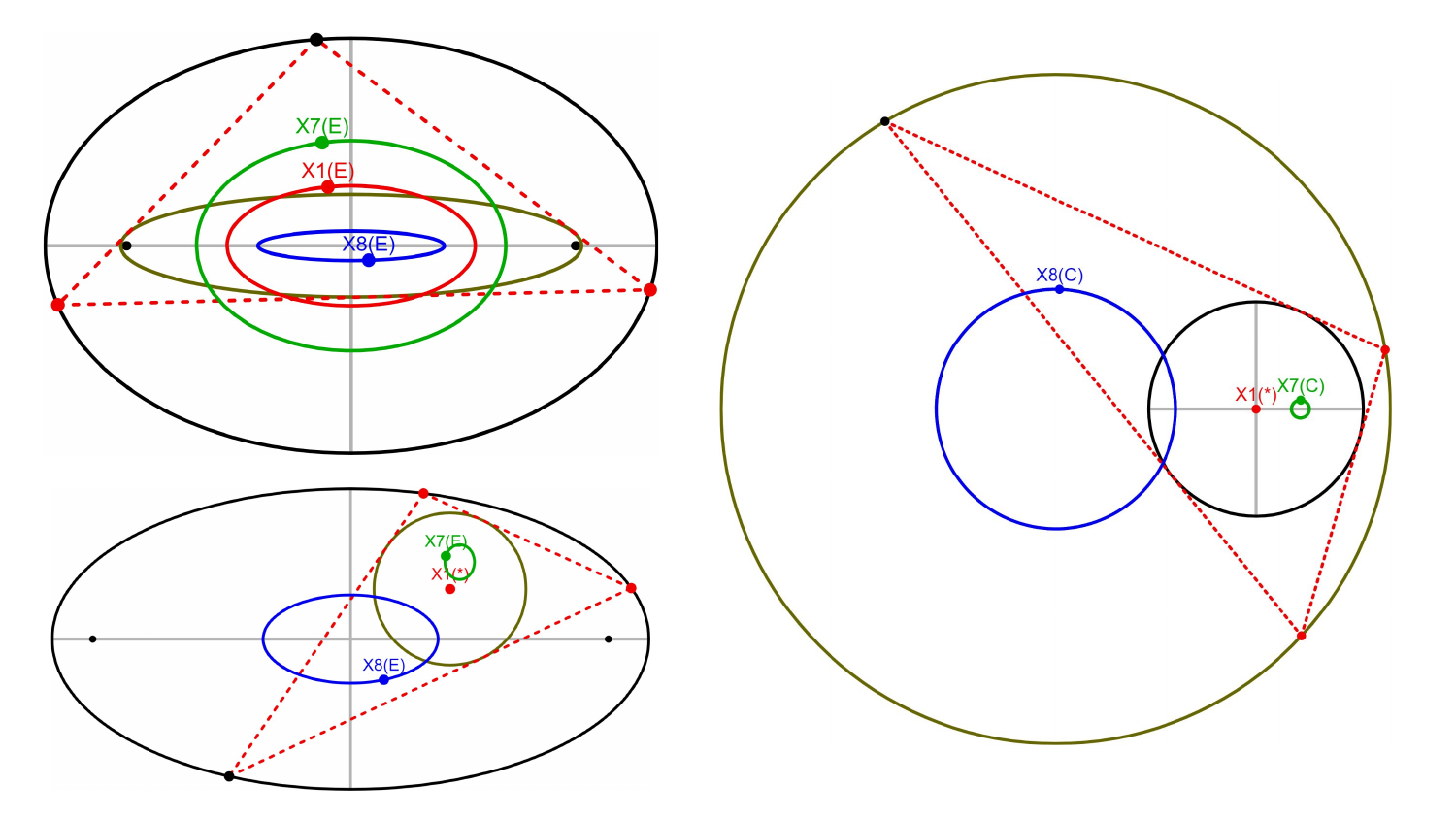}
\caption{Three Poncelet families for which the loci of $X_k$, $k=1,7,8$ are conics: (i) \tb{top left}: confocal pair. Live: \hrefs{https://bit.ly/46aRmHN}; (ii) \tb{bottom left}: generic circular caustic, all are ellipses, marked `(E)'. Live: \hrefs{https://bit.ly/410Dj3Y}; (iii) \tb{right}: Chapple's porism (`bicentric' triangles). Live: \hrefs{https://bit.ly/4kGRYbC}. Note that in (ii) and (iii), the incenter $X_1$ is stationary, marked `(*)', at the center of the caustic. Also note that in (iii), the loci of the Gergonne center $X_7$ and Nagel's point $X_8$ are circles, marked `(C)'.} 
\label{fig:x1x7x8-conic}
\end{figure}

\begin{figure}
\centering
\includegraphics[width=\linewidth]{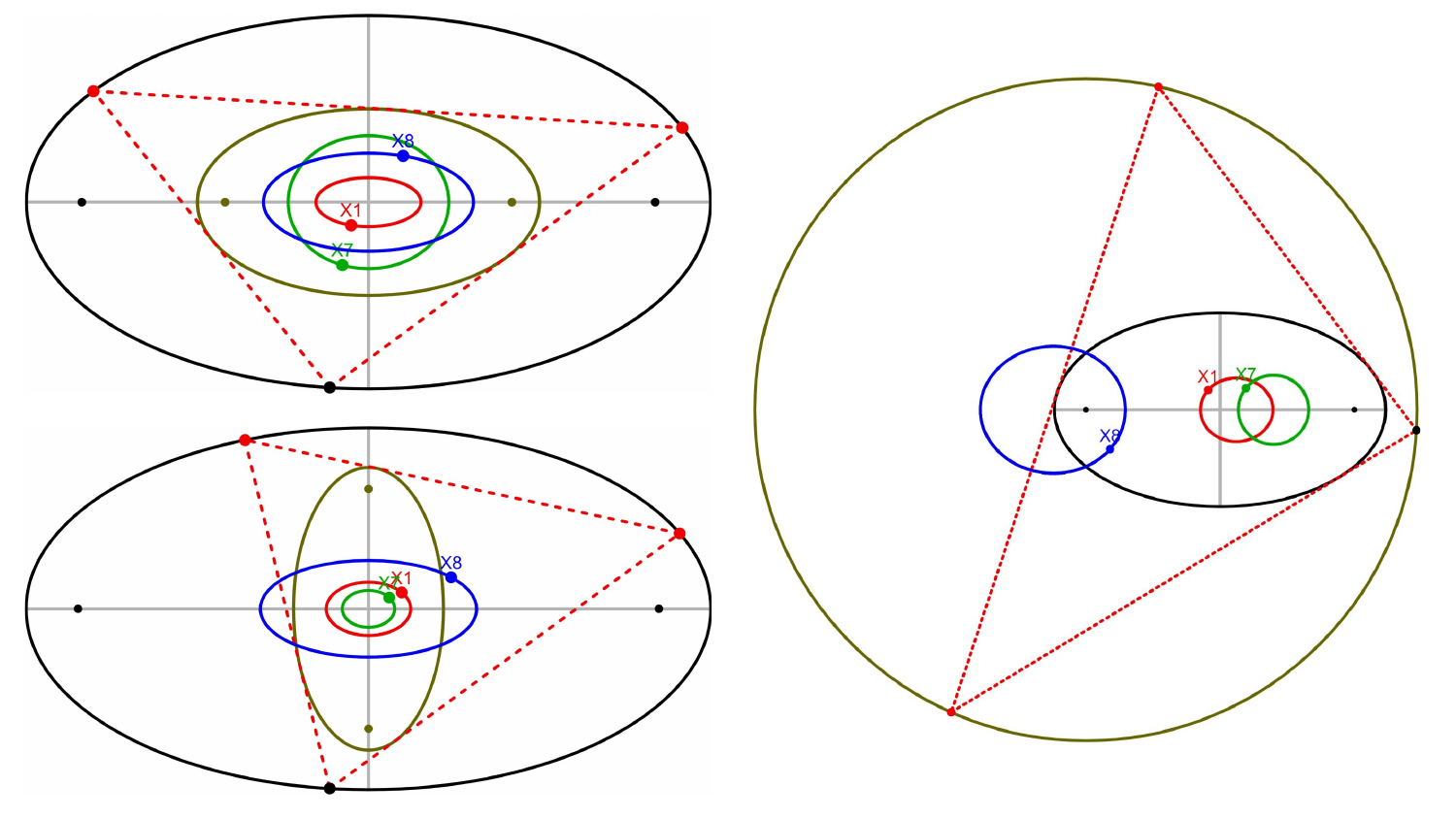}
\caption{Three Poncelet families for which the loci of $X_k$, $k=1,7,8$ are non-conics: (i) \tb{top left}: pair of homothetic ellipses. Live: \hrefs{https://bit.ly/3GQHgkW}; \tb{bottom left}: pair of `dual' ellipses. Live: \hrefs{https://bit.ly/40XaxRK}; (iii) \tb{right}: the MacBeath family. Live: \hrefs{https://bit.ly/3Gt6xSk}}
\label{fig:x1x7x8-nonconic}
\end{figure}

\subsection{Locus behaviors}
\label{sec:behaviors}

\begin{samepage}
\begin{proposition}
Over Poncelet triangles with a circular caustic, from $X_1$ to $X_{1000}$, $X_k$ sweep a conic locus for the following $k$:
\tb{2}, \tb{3}, \tb{4}, \tb{5}, 7, 8, 10, 11, 12, \tb{20}, 35, 36, 40, 42, 46, 55, 56, 57, 65, 79, 80, 81, \tb{140}, 145, 165, 171, 174, 176, 202, 213, 226, 354, 355, 365, 366, 367, 368, 370, \tb{376}, \tb{381}, \tb{382}, 388, 390, 396, 482, 484, 495, 496, 497, 498, 499, 506, 507, \tb{546}, \tb{547}, \tb{548}, \tb{549}, \tb{550}, 551, 553, 554, 559, 588, 590, 593, 597, 605, 609, 611, 612, \tb{631}, \tb{632}, 938, 939, 940, 941, 942, 944, 946, 950, 954, 955, 962, 975, 977, 980, 981, 989, 999, 1000.
\end{proposition}
\end{samepage}

Note that boldfaced entries above are triangle centers which are fixed combinations of the barycenter $X_2$ and the circumcenter $X_3$ and will, over any Poncelet pair, sweep a conic locus \cite{helman2021-theory}. 

Amongst some of the above-listed centers, six types of ``locus behaviors'' are observed: 

\begin{enumerate}
\item \tb{$\E$-homothety}: $X_k$, $k=2$ (see \cite{sergei2016-com}), $4$, $8$ (concentric), $10$, $145$, $368$, $370$, $551$, $944$, $946$ are homothetic to $\E$. 
\item \tb{Axis alignment}: $X_k$, $k=7$, $79$, $390$, $1000$ are axis-aligned with $E$ but not homothetic to it.
\item \tb{C-focus}: $X_k$, $k=5$, $12$, $355$, $495$, $496$ have $C$ for a focus.
\item \tb{C-major}: the major axis of $X_k$, $k=3$ (foci on $E$'s axes), $35$, $40$, $55$, $165$ pass through $C$.
\item \tb{C-minor}: the minor axis of $X_k$ $k=46$, $56$, $57$, $999$ pass through $C$. 
\item\tb{Circle}: $X_k$, $k=11$ (on incircle), $36$, $65$, $80$ (centered on C), $354$, $484$, $942$ are circles.
\end{enumerate}

Pictures of the above can be found in \cite{reznik2024-nifty-loci}. Do refer to \cite{garcia2025-four-families} where properties of four special triangle Poncelet porisms are described with a circular caustic. These keep certain key triangle centers stationary, namely, the incenter $X_1$ (on a focus of $\E$), the barycenter $X_2$, the orthocenter $X_4$, and the Gergonne point $X_7$. We also identify key conservations, some of which generalize to $n>3$ Poncelet porisms.

\section{An equilateral prompts degeneracies}
\label{sec:degenerate}
This section describes many degeneracies brought about by the `presence' of an equilateral triangle in $\T_{\circ}$. As before, refer to \cref{tab:centers} for definitions of any triangle centers $X_k$ mentioned.

The locus $\Et$ of centroids of a 1d family of equilateral triangles inscribed in an ellipse $\E=(a,b)$ is also an ellipse, concentric and axis-aligned with $\E$, and with semi-axes $a_{\triangle},b_{\triangle}$ given by \cite{stanev2019-equi} (see animations in \cite{hui2018-equi}, where these are called $a_1,b_1$):

\[ a_{\triangle}=\frac{a\,c^2}{a^2 + 3b^2}\,\rc\quad b_{\triangle}= \frac{b\,c^2}{3a^2 +  b^2}\cdot\]

Let $\Tt$ denote the family of Poncelet triangles inscribed in $\E$ and about a circle centered at $C=[a_{\triangle}\cos t, b_{\triangle} \sin t]$ on $\Et$. By continuity, $\Tt$ contains an equilateral triangle. Let $V_{\triangle}=[a \cos{t_{\triangle}},b\sin{t_{\triangle}}]$ on $\E$ be a vertex of said equilateral triangle. Using CAS simplification:
\begin{proposition}
$t_{\triangle}$ is given by the three solutions of:
\begin{align*}
& (a^2 + 3b^2)\left[(3a^2 + b^2)\cos^2t_{\triangle} - 2b^2\sin^2t - 2b^2\sin t\sin t_{\triangle}\right] \\
&+ (3a^2 + b^2)\left[-2a^2\cos t\cos t_{\triangle} + (a^2 - 3b^2)\cos^2 t\right]=0 \ldotp
\end{align*}
\label{prop:equi-vtx}
\end{proposition}
Referring to \cref{fig:equi-vtx}, we leave it as an exercise: show that a Poncelet triangle will be equilateral when the triangle $V_{\triangle} C Z$ is isosceles, with $Z=[a\cos t,-b\sin t]$.

Note: in \cref{prop:x11} $Z=X_{80}$.

\begin{figure}
\centering
\includegraphics[width=0.7\linewidth]{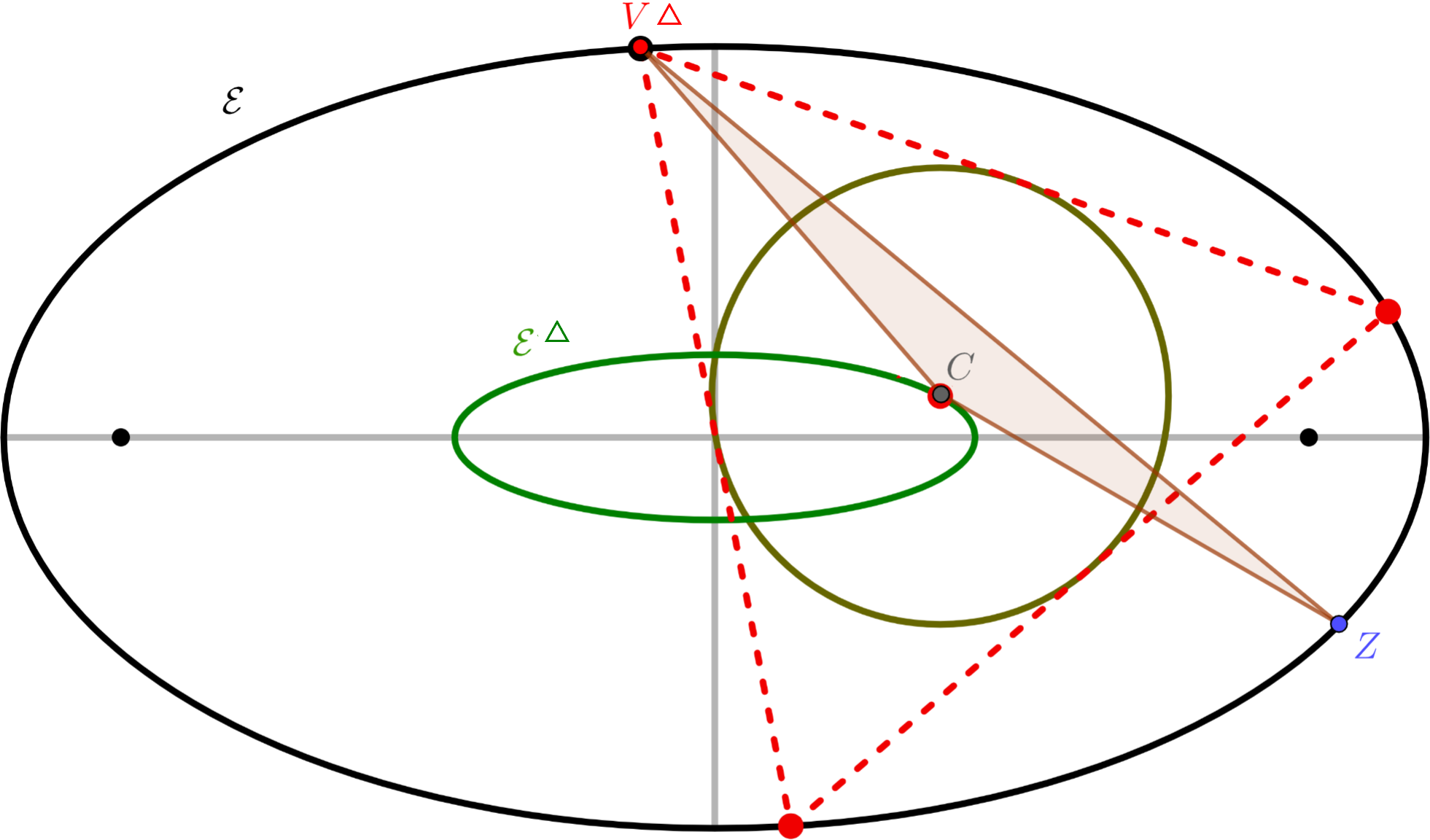}
\caption{The Poncelet equilateral triangle (dashed red) used in the proof for \cref{prop:equi-vtx}. $\Et$ (dark green) is the locus of centroids of $\E$-inscribed equilateral triangles. The triangle $V_{\triangle} C Z$ is an isosceles. Live: \hrefs{https://bit.ly/4eTuU8i}}
\label{fig:equi-vtx}
\end{figure}

Sample loci of $X_k$, $k=2,3,4,5$ over $\Tt$ are shown  \cref{fig:loci-equi} (top). Notice they simultaneously pass through the incenter $X_1$ when the Poncelet triangle becomes equilateral. In \cref{fig:loci-equi} (bottom) are shown the loci of  $k=7,8,11,80$. The first two display the same behavior, while the latter two are stationary.

\begin{figure}
\centering
\includegraphics[width=.8\linewidth]{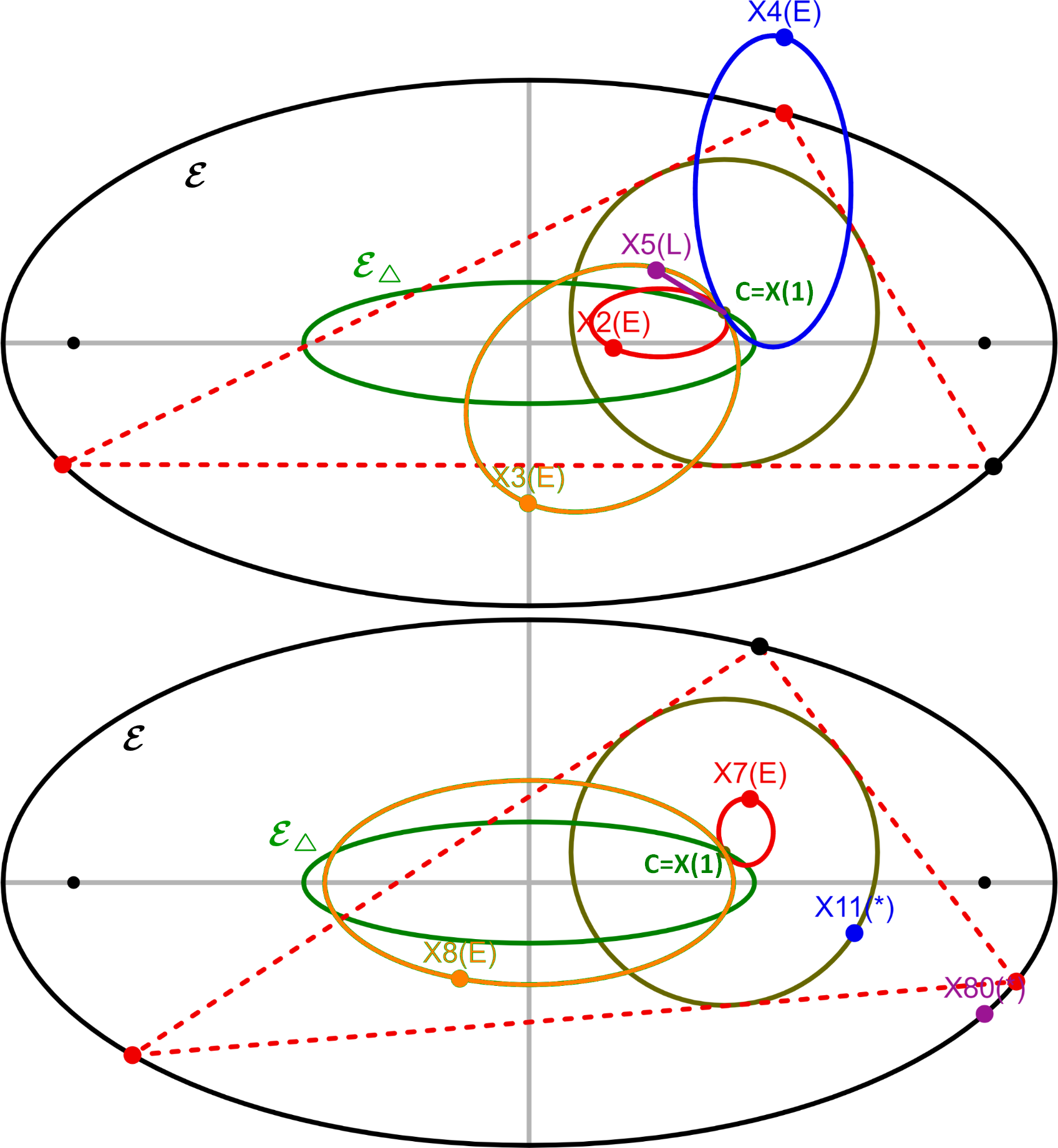}
\caption{\tb{top}: when the incenter $C=X_1$ is on $\Et$ (dark green), the loci of $X_k$, $k=2,3,4,5$ (red, light orange, blue, purple), pass simultaneously through the incenter when the equilateral triangle in the family is traversed, notice that $\L_5$ (purple) has collapsed to a line segment, indicated above as $X5(L)$. Live: \hrefs{https://bit.ly/46g8mwf}; \tb{bottom}: same situation, showing the loci of $X_k$, $k=7,8,11,80$ (red, light orange, blue, purple). The latter two are stationary on the incircle and $\E$, respectively, and along the $X_1 X_5$ line. Live: \hrefs{https://bit.ly/3Gxebeq}}
\label{fig:loci-equi}
\end{figure}

\subsection{\torp{C is a vertex of $\L_3$}{C is a vertex of L(3)}}
Referring to \cref{fig:einf-out-in,fig:einf-on-ctr}:

\begin{corollary}
Over $\Tt$, $C$ is a (major) vertex of $\L_3$. If $C$ is interior (resp. exterior) to $\Et$, it is interior (resp. exterior) to $\L_3$. 
\end{corollary}

\begin{proof}
This follows from \cref{prop:x3-foci}: the major axis of $\L_3$ passes through $C$.
\end{proof}


\begin{figure}
\centering
\includegraphics[width=\linewidth,trim={0 1.3in 0in .6in},clip]{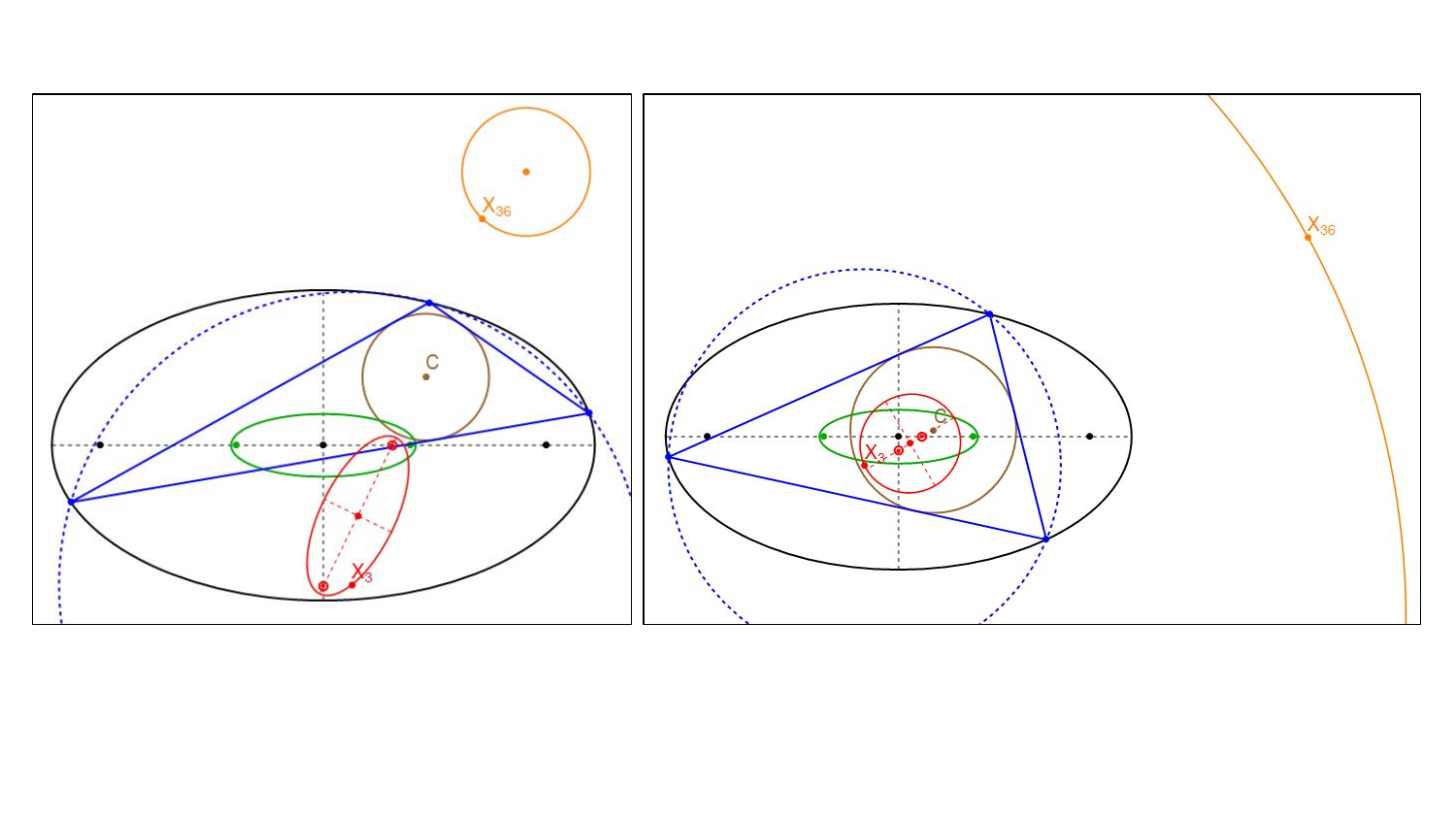}
\caption{On the left (resp. right), $C=X_1$ is outside (resp. inside) $\Et$ (green). In this case, $C$ is exterior (resp. interior) to $\L_3$ (red), and $\L_{36}$ (orange) is disjoint with (resp. contains) $\E$.}
\label{fig:einf-out-in}
\end{figure}

\begin{figure}
\centering
\includegraphics[width=\linewidth,trim={0 1.25in 0in .65in},clip]{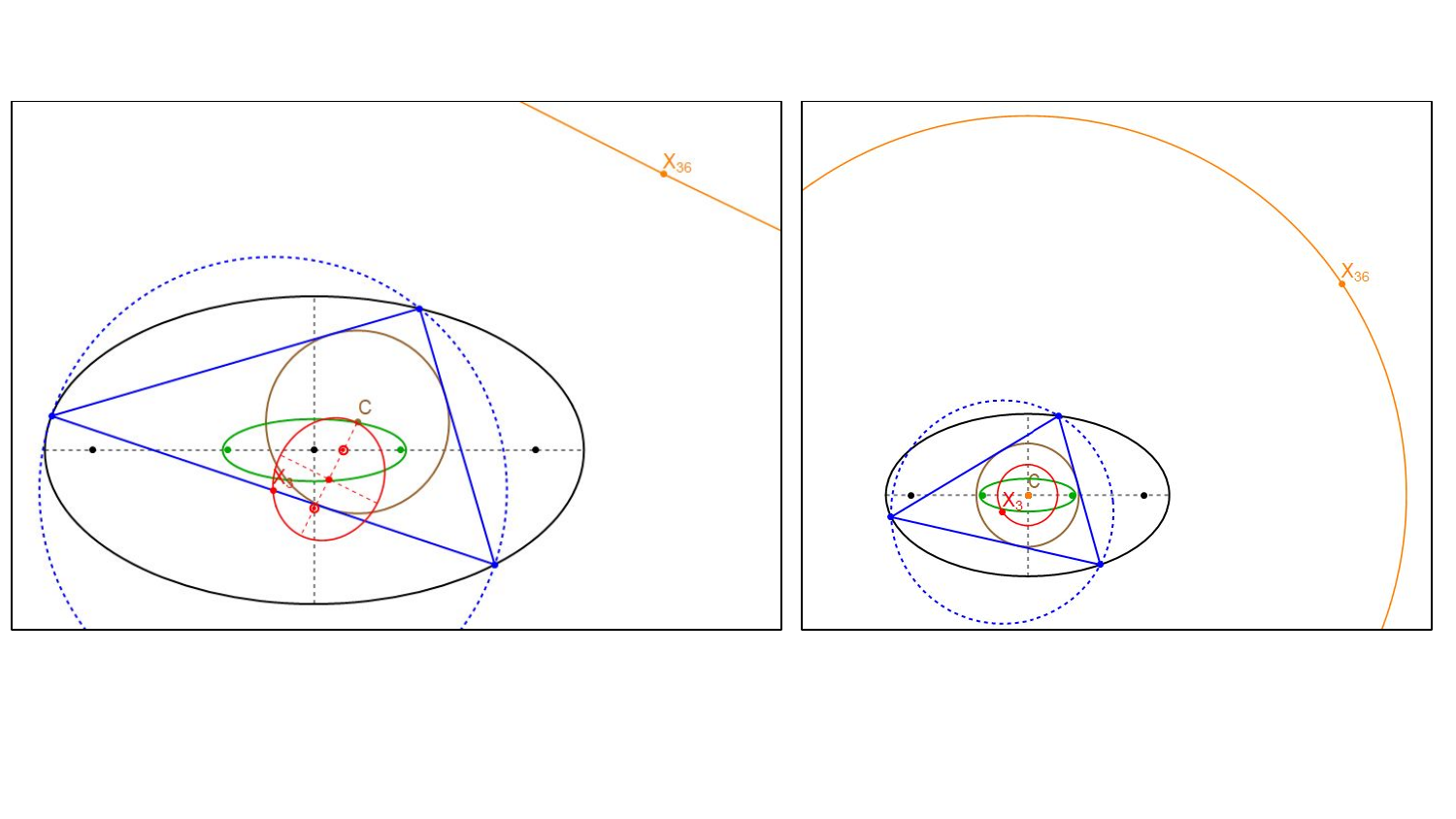}
\caption{On the left (resp. right), $C=X_1$ is on (resp. at the center of) $\Et$ (green). In this case, $C$ is a vertex (resp. at the center) of $\L_3$ (red), and $\L_{36}$ (orange) is a line (resp. a circle concentric with $\E$).}
\label{fig:einf-on-ctr}
\end{figure}

\subsection{\torp{Invariant eccentricity: $\L_3$, $\L_7$}{Invariant eccentricity: L(3), L(7)}}

\begin{proposition}
Over all $C$ on $\Et$, the aspect ratio of $\L_3$ is invariant and given by:
\[ \frac{a_3}{b_3}=\frac{\left(3 a^2+b^2\right) \left(a^2+3 b^2\right) \left(a^4-b^4\right)}{c^2 \left(6 a^5 b+20 a^3 b^3+6 a b^5\right)}\cdot\]
\end{proposition}

\begin{proof}
Direct from \cref{eqn:a3b3} in \cref{prop:x3-foci} for $[x_c,y_c]$ on $\Et$.
\end{proof}

\begin{proposition}
Over all $C$ on $\Et$, the aspect ratio $\L_7$ is invariant and given by:
\[ \frac{a_7}{b_7} ={{\frac {2\,{a}^{2}+{b}^{2}}{\sqrt {2\,{a}^{4}+5\,{a}^{2}{b}^{2}+2
\,{b}^{4}}}}}\cdot\]
\end{proposition}

\begin{proof}
Direct from \cref{prop:x7-locus}.
\end{proof}

\subsection{\torp{Segment $\L_5$}{Segment L(5)}}

\begin{proposition}
Over all $C=[a_{\triangle} \cos{t},b_{\triangle} \sin{t}]$ on $\Et$, the minor axis of $\L_5$ collapses to zero ($b_5=0$) and the latter becomes a segment  of length $l_5$ with an endpoint on $C$ and the other at $F'_{5,{eq}}$. These are given by: \begin{align*}
l_5^2 &= \frac{c^8 \left(a^6+15 a^4 b^2+15 a^2 b^4+b^6-c^6 \cos (2 t)\right)}{16 \left(3 a^5 b+10 a^3 b^3+3 a b^5\right)^2}\rc\\
F'_{5,{eq}} &= \left[\frac{\left(a^4-b^4\right) \cos{t}}{2 \left(a^3+3 a b^2\right)},\frac{\left(a^4-b^4\right) \sin{t}}{2 \left(3 a^2 b+b^3\right)}\right]\cdot
\end{align*}
\end{proposition}

\begin{proof}
Direct from \cref{prop:a5b5} with $[x_c,y_c]$ on $\Et$.
\end{proof}

In \cite{etc}, $X_{12}$ is the internal center of similitude of the incircle and the Euler circle. 

\begin{corollary}
Over $\Tt$, $\L_{12}$ is a segment with one endpoint at $C$.
\end{corollary}

\begin{proof}
This must be the case because $\L_5$ (a segment) passes through $C$.
\end{proof}

\subsection{\torp{Stationary Feuerbach's point $X_{11}$, $X_{80}$}{Stationary X(11), X(80)}}

In \cite{etc}, $X_{80}$ is the reflection of the incenter $X_1$ on the Feuerbach point $X_{11}$. Referring to \cref{fig:x59-cut}, CAS simplification yields:

\begin{proposition}
\label{prop:x11}
Over $\Tt$, Feuerbach's point $X_{11}$ (resp.\@~$X_{80}$) is stationary at the non-$C$ intersection of $\L_5$ (a segment) with the incircle (resp. with $\E$). With $C=[a_{\triangle} \cos{t},b_{\triangle} \sin{t}]$ on $\Et$, these are given by:
\begin{align*}
X_{11} & = \frac{a^2+b^2}{c^2}\left[a_{\triangle}\cos t, -b_{\triangle}\sin t\right], \\
X_{80} & = \left[a\cos t, -b\sin t\right]\ldotp
\end{align*}
\end{proposition}

Note that $X_k,k=1$, $5$, $11$, $12$, $80$ are points on the $X_1 X_5$ line, for a complete list see \cite{etc-central-lines}. Also note that $X_{80}$ is $Z$ in the proof of \cref{prop:equi-vtx}.

\begin{definition}[Isogonal conjugate]
\label{def:isog}
The isogonal conjugate $P^{\dagger}$ of a point $P$ with respect to a triangle $T$ is the point of concurrence of the cevians of $P$ reflected upon the angle bisectors. If the barycentrics of $P$ are $[z_1:z_2:z_3]$,  $P^{\dagger}=[l_1^2/z_1 : l_2^2/z_2 : l_3^2/z_3]$, where the $l_i$ are the sidelengths.
\end{definition}

$X_{106}$ is the isogonal conjugate of $X_1 X_2$ with the line at infinity.

\begin{observation}
Over $\Tt$, $\L_{106}$  is a circle passing through $X_{80}$.
\end{observation}

\begin{lemma}
Let $C=[a_{\triangle} \cos{t}, b_{\triangle} \sin{t}]$. Over $\Tt$, the intersection $\It=[x_{\triangle},y_{\triangle}]$ of the ray $X_1 X_5$ with $\E$ is given by:
\begin{align*}
x_{\triangle}=&
\frac{\left(a^{2} v \left(a^{4}-3 a^{2} b^{2}-2 b^{4}\right)  \sin^{2}{t}-b^{4} u^{2}  \cos^{2}{t}\right)  a \cos{t}}{b^{4} u^{2}  \cos^{2}{t} +a^{4} v^{2}  \sin^{2}{t}}\rc\\
y_{\triangle}=& \frac{  \left(b^{2} u \left(2 a^{4}+3 a^{2} b^{2}-b^{4}\right)  \cos^{2}{t} +a^{4} v^{2}  \sin^{2}{t}\right) b \sin{t}  }{b^{4} u^{2}  \cos^{2}{t} +a^{4} v^{2}  \sin^{2}{t}}\rc
\end{align*}
where $u=(3\,a^2+b^2)$ and $v=(3\,b^2+a^2)$.
\label{lem:isosceles}
\end{lemma}

Referring to \cref{fig:isosceles}, an interesting property of $\It$ is:

\begin{proposition}
\label{prop:isosceles}
In $\Tt$, the Poncelet triangle $\It B_1 B_2$ is isosceles, unique in the family (other than the equilateral triangle), with base $B_1 B_2$, the chord of $\E$ tangent to the incircle at the stationary Feuerbach's point $X_{11}$. 
\end{proposition}

\begin{proof}
The triangle is isosceles since $\It,X_1,X_5$ are collinear. Poncelet's theorem implies that the base is tangent to the incircle at Feuerbach's point $X_{11}$, since the latter is on $X_1 X_5$ \cite{etc}. We omit the expressions for $B_1$ and $B_2$ since they are rather long. To explain why the isosceles triangle is unique, consider that $\L_2$ (see \cref{prop:locus-x2}) will pass through the incenter $X_1$ (at the equilateral configuration) and intersect $X_1 X_5$ at another single point $\It$, consistent with the fact that all triangle centers of an isosceles triangle must lie on its axis of symmetry, see \cref{fig:isosceles-unique}.
\end{proof}

\begin{figure}
\centering
\includegraphics[width=0.8\linewidth]{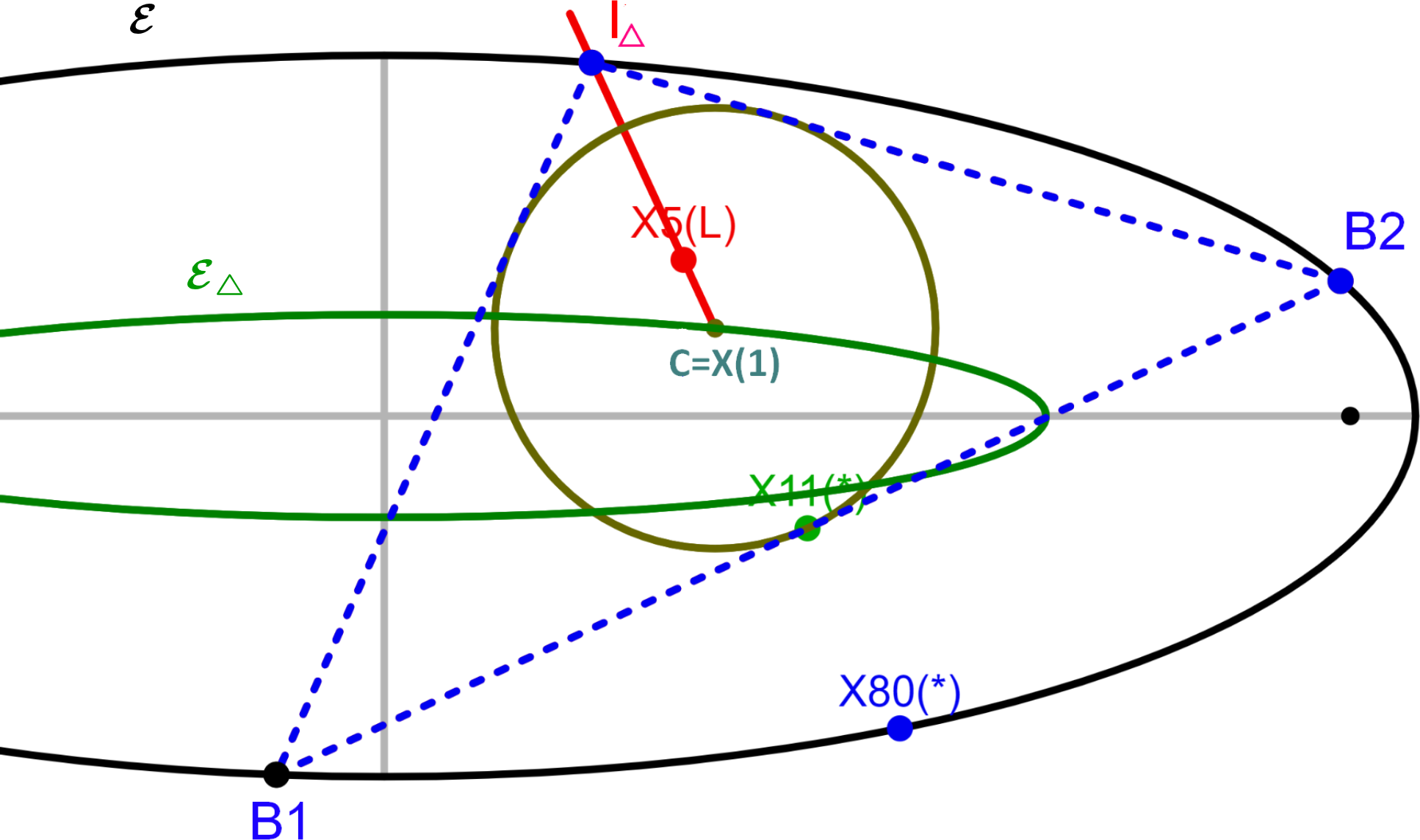}
\caption{Construction for  \cref{prop:isosceles}: $C=X_1$ on $\Et$ (dark green), the Poncelet triangle has a vertex on the intersection $\It$ of the ray $X_1 X_5$ with $\E$. The said triangle is isosceles, and its base $B_1 B_2$ is the chord of $\E$ touching the incircle at Feuerbach's point $X_{11}$. Live: \hrefs{https://bit.ly/44yOUJW}}
\label{fig:isosceles}
\end{figure}

\begin{figure}
\centering
\includegraphics[width=\linewidth]{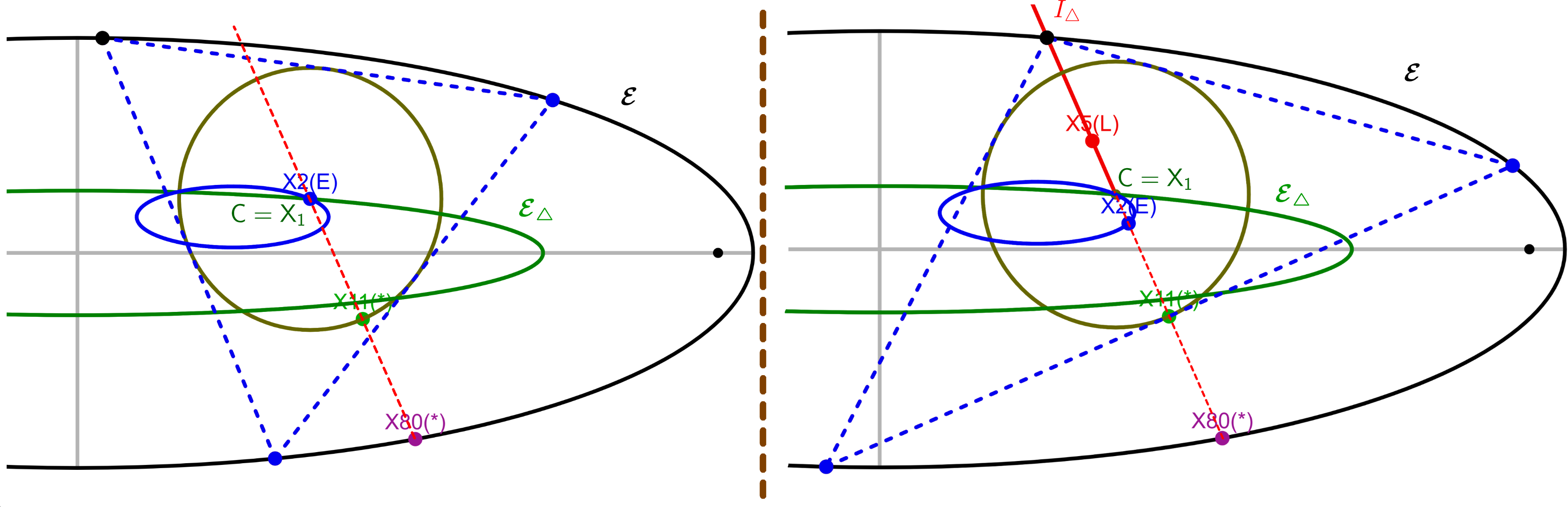}
\caption{$C$ is on $\Et$ (dark green). \tb{left}: the barycenter $X_2$ is on $C=X_1$, i.e., the triangle is equilateral; \tb{right}: the barycenter $X_2$ is on the other intersection of $X_1 X_5$ with its locus, in which case the triangle is isosceles (multiple triangle center lie along its axis of symmetry), as in \cref{fig:isosceles}.}
\label{fig:isosceles-unique}
\end{figure}

\subsection{\torp{Line $\L_{36}$}{Line L(36)}}

In \cref{prop:x36} we saw that $\L_{36}$ was a circle. Nevertheless, and referring to \cref{fig:einf-on-ctr} (left):

\begin{proposition}
\label{prop:x36-inf}
Over $\Tt$, 
with
$C=[x_c,y_c]$ on $\Et$, $\L_{36}$ degenerates to the line:
\[ b^{2} x_c x + a^{2} y_c y = \frac{b^{2} \left(a^{4}+2 a^{2} b^{2}+5 b^{4}\right) x_c^{2}+ a^{2}\left(5 a^{4}+2 a^{2} b^{2}+b^{4}\right) y_c^2}{c^4}\rd\]
\label{prop:x36-line}
\end{proposition}
\begin{proof}
Direct from   \cref{prop:x36}, taking the limit when $[x_c,y_c]$ is on $\Et$ and simplifying.
\end{proof}

Referring to \cref{fig:einf-out-in}:

\begin{observation}
Over $\Tt$, $\L_{36}$ is disjoint with (resp. contains) $\E$ if $C$ is exterior (resp. interior) to $\Et$.
\end{observation}

Referring to \cref{fig:x3x36-min}:

\begin{proposition}
If $C=[x_c,y_c]$ is on $\Et$, the major axis of $\L_3$ is perpendicular to $\L_{36}$ (circle degenerates to an infinite line). These meet at: 
\[ X^\perp_{36}=c^{-4}\left[x_c (a^4 + 2 a^2 b^2 + 5 b^4), y_c (b^4+2 b^2 a^2 + 5 a^4)\right]\cdot \]
\end{proposition}

\begin{proof}
Direct from \cref{prop:x3-foci}, which gives $a_3$ (major axis of $\L_3$), and \cref{prop:x36-line}. 
\end{proof}

\begin{figure}
\centering
\includegraphics[width=0.8\linewidth]{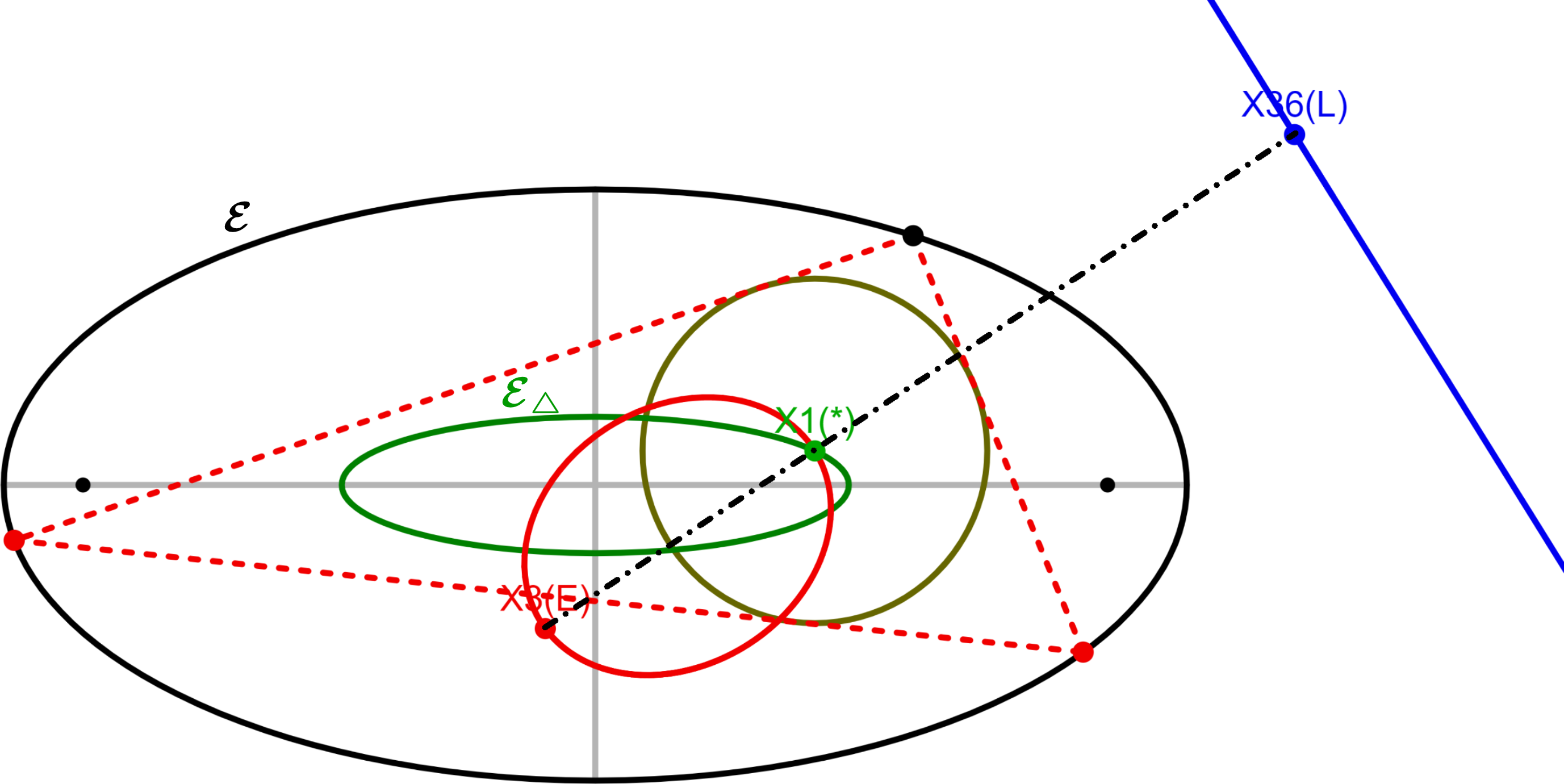}
\caption{$C=X_1$ is on $\Et$ (dark green). When the circumcenter $X_3$ is on the non-$X_1$ major vertex of its locus (red), $X_3 X_1$ is perpendicular to $\L_{36}$ (infinite line). At this moment, $|X_1 X_{36}|$ is minimal. Live: \hrefs{https://bit.ly/4kOcYNQ}}
\label{fig:x3x36-min}
\end{figure}

Referring to \cref{fig:x3x36-circum} (right):

\begin{observation}
Over $\Tt$, when the circumcenter $X_3$ is on $C=X_1$, i.e., at a major vertex of its elliptic locus, $X_{36}$ is on the line at infinity.
\end{observation}

Referring to \cref{fig:x3x36-circum}, we specialize \cref{cor:x3x36} to its degenerate form:

\begin{proposition}
Over $\Tt$, when the circumcenter $X_3$ is on the non-$X_1$ major vertex of its locus, the distance $|X_1 X_{36}|$ is minimal and given by:
\[ |X_1 X_{36}|_{min} =\frac{4\left(a^{2}+b^{2}\right)  \sqrt{a^{4} y_c^{2}+b^{4} x_c^2}}{c^4}\rd \]
Furthermore, the circumradius $R$ is minimized (resp. maximized) when $X_3=X_1$ (resp. the circumcenter $X_3$ is on the non-$X_1$ vertex of its locus).
\begin{align*}
R_{min} &= 2\left( \frac{b\sqrt{a^4-c^2 x_c^2}- a\sqrt{b^4+c^2 y_c^2}}{c^2} \right)\cdot\\
R_{max} & =\frac{\left(a^{2}+b^{2}\right) \sqrt{a^{4}+6 a^{2} b^{2}+b^{4}}\, \sqrt{a^{4} y_c^{2}+b^{4} x_c^2}}{2 a^{2} b^{2} c^2}\rd
\end{align*}
\end{proposition}

\begin{figure}
\centering
\includegraphics[width=\linewidth]{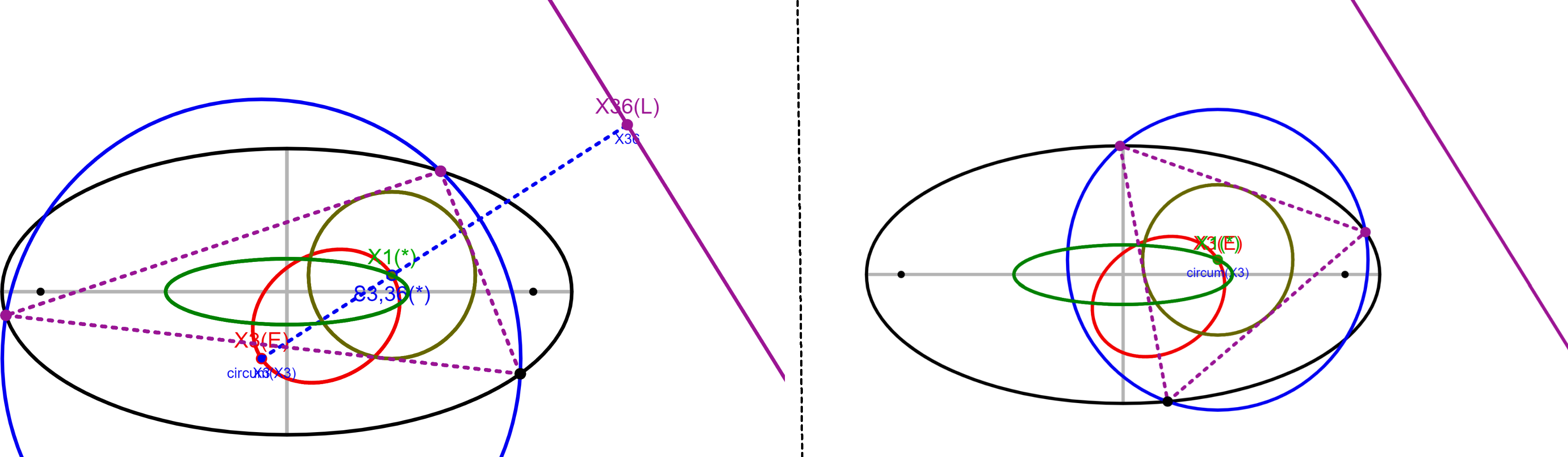}
\caption{$C=X_1$ is on $\Et$ (dark green). \tb{left}:When $X_{3}$ is on the non-$X_1$ vertex of its locus, the circumradius $R$ is maximized and $|X_1 X_{36}|$ is minimal; \tb{right}: if $X_{3}$ on the incenter $X_1$, the triangle is equilateral, $R$ is minimized while $|X_1 X_{36}|$ is `maximal', since $X_{36}$ is on the line at infinity. Live: \hrefs{https://bit.ly/4lZZkYX}}
\label{fig:x3x36-circum}
\end{figure}

Referring to \cref{fig:x36-env}:

\begin{observation}
As $C$ sweeps $\Et$, the line-locus $\L_{36}$ envelops a degree-6 closed algebraic curve.
\end{observation}

\begin{figure}
\centering
\includegraphics[width=0.7\linewidth]{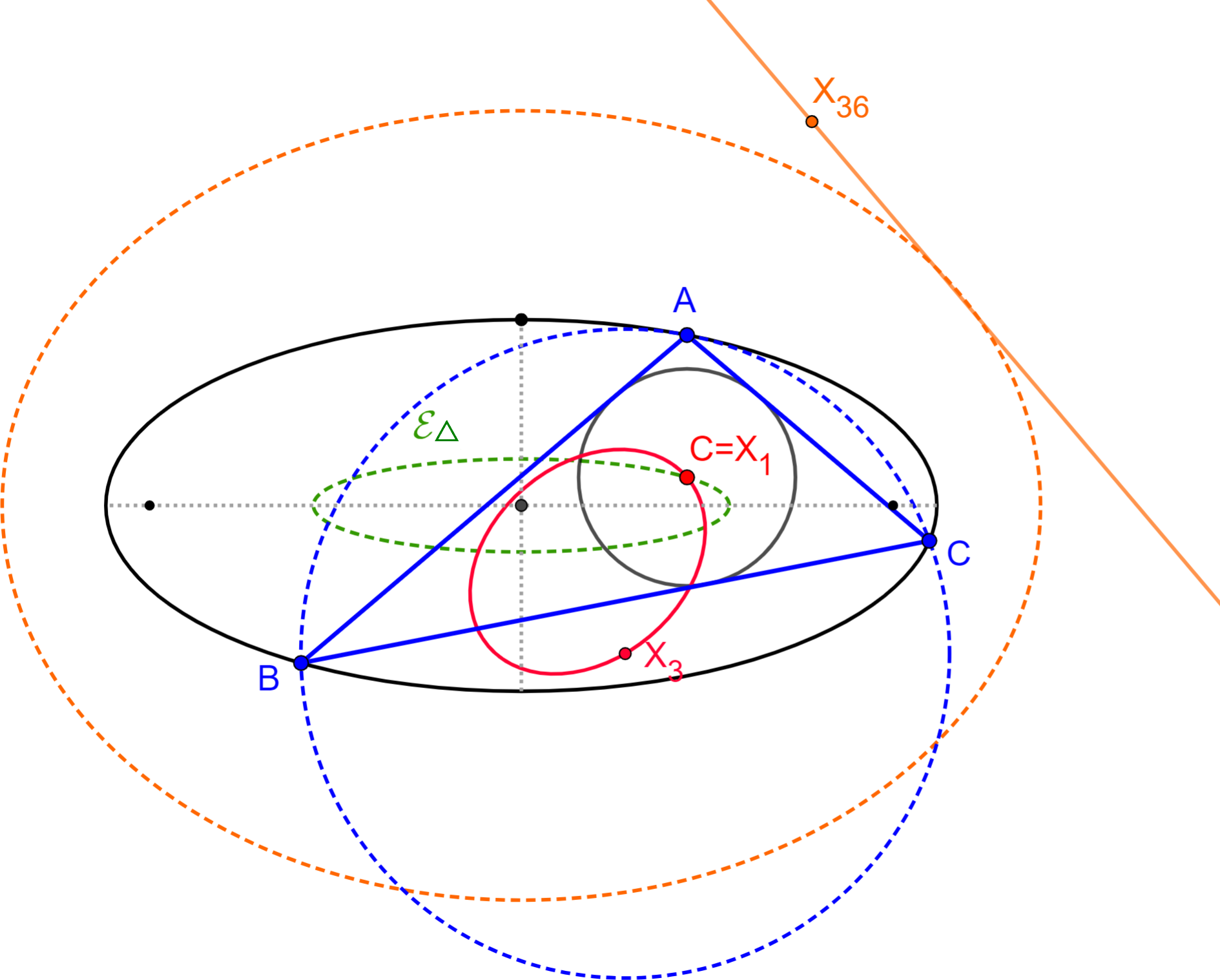}
\caption{As $C$ moves along $\Et$ (dashed green), the line-locus $\L_{36}$ (solid orange) envelops a degree-6 closed algebraic curve (dashed orange).}
\label{fig:x36-env}
\end{figure}

\subsection{\torp{Elliptic $\L_{59}$}{Elliptic L(59)}}

In \cite{etc}, $X_{59}$ is the isogonal conjugate of Feuerbach's point $X_{11}$.  Referring to \cref{fig:x59-locus}, in several Poncelet families previously studied, $\L_{59}$ is either an oval or a self-intersected curve, e.g., see \cite[Fig.2]{garcia2020-ellipses} and \cite[Fig.9]{reznik2020-ballet}. Nevertheless, it collapses to a conic in the following two cases. Firstly:

\begin{proposition}
In Chapple's porism ($\E$, $\E_c$ are circles), $\L_{59}$ is an ellipse with major axis on $X_1 X_3$, with semi-axis lengths given by:
\[ a_{59,chapple}=R,\,\,\,b_{59,chapple} = \frac {R\sqrt {{R}^{2}-{d}^{2}}}{\sqrt {9\,{R}^{2}-{d}^{2}}}\rc \]
where $R$ is the radius of $\E$ and $d$ is the distance between the centers of $\E$ and $\E_c$.
\end{proposition}

\begin{figure}
\centering
\includegraphics[width=\linewidth]{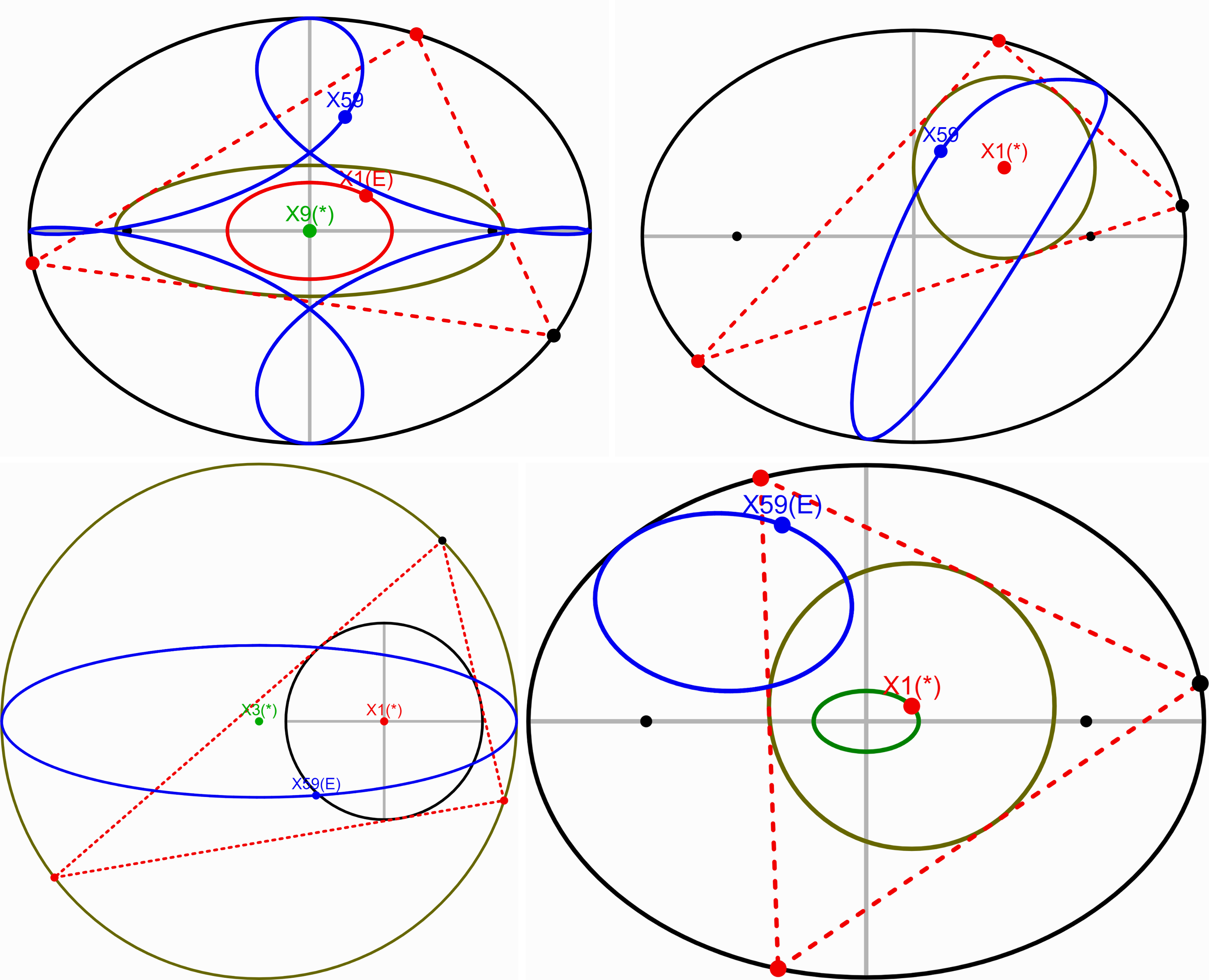}
\caption{Four faces of $\L_{59}$: \tb{top left}: self-intersected in the confocal pair; \tb{top right}: a non-conic under a generic incircle porism; \tb{bottom left}: an ellipse under Chapple's porism; \tb{bottom right}: an ellipse, internally tangent to $\E$ in an incircle porism with $X_1$ on $\Et$ (dark green), see  \cref{cor:x59-ell}.}
\label{fig:x59-locus}
\end{figure}

As above, let $P^{\dagger}$ denote the isogonal conjugate of a point $P$ with respect to a triangle $T$. In \cite{skutin2013-isogonal} it is shown that the locus of $P^{\dagger}$ over circle-inscribed Poncelet triangles is a circle. As above, let $\T$ be a family of Poncelet triangles interscribed between two generic conics $\E,\E_c$.

Referring to \cref{fig:isog-locus}, the following is proved in \cite[Thm.2]{garcia2025-x4-conjugate}:

\begin{lemma}
Over $\T$, the locus of $P^{\dagger}$ is a conic.
If $P$ is on $\E_c$
the locus is an ellipse interior to $\E$ and touching it at a point. If $P$ is on $\E$ the locus is the union of a straight line and the line at infinity (degenerate hyperbola).
\end{lemma}

An alternative proof to 
\cref{prop:x36-inf} is to based on the fact that $X_{36}$ and $X_{80}$ are isogonal conjugates \cite{etc}:

\begin{corollary}
Over $\Tt$, $\L_{36}$ is a line since $X_{80}$ is stationary on $\E$, \cref{prop:x11}.
\end{corollary}

\begin{figure}
\centering
\includegraphics[width=\linewidth]{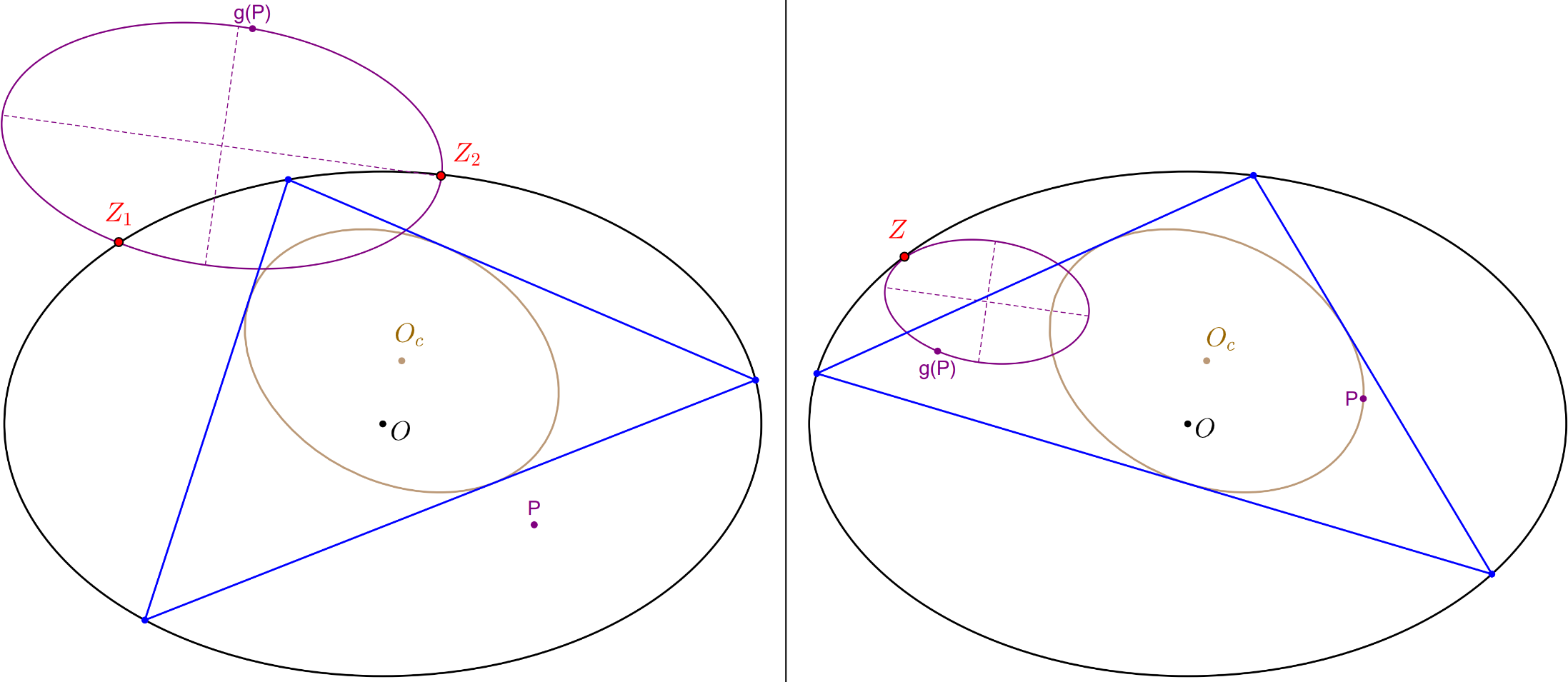}
\caption{As show in \cite{garcia2025-x4-conjugate}, over a generic Poncelet triangle family, the locus of the isogonal conjugate $P^{\dagger}$ of a fixed point $P$ is a conic (purple). \tb{left}: $P$ is interior to $\E$ and exterior to $\E_c$, so the locus of $P^{\dagger}$ will cross $\E$ at two points $Z_1,Z_2$. \tb{right}: if $P$ is on the caustic, the locus of $P^{\dagger}$ is an ellipse tangent to $\E$ at a point $Z$. Video: \url{https://youtu.be/v_K0xoQy4IM}}
\label{fig:isog-locus}
\end{figure}

Secondly, referring to \cref{fig:x59-locus} (bottom right) and \cref{fig:x59-cut}:

\begin{corollary}
\label{cor:x59-ell}
Over $\Tt$, $\L_{59}$ is an ellipse interior to $\E$ and touching it at a point.
\end{corollary}

\begin{proof}
$X_{59}$ is the isogonal conjugate of Feuerbach's point $X_{11}$ \cite{etc}, which is stationary on the incircle over $\Tt$.
\end{proof}

Referring to \cref{fig:x59-cut}, let $\It$ denote the point of tangency between $\L_{59}$ (an ellipse over $\Tt$) and $\E$.

\begin{observation}
When the Poncelet triangle is isosceles, i.e., one of its vertices is $\It$, $X_{59}$ is at $\It$. As the Poncelet triangle approaches the equilateral configuration, $X_{59}$ approaches the major vertex of its locus. Since $X_{59}$ is undefined for an equilateral triangle, $\L_{59}$ is an ellipse minus one major vertex.
\end{observation}

\begin{figure}
\centering
\includegraphics[width=.7\linewidth]{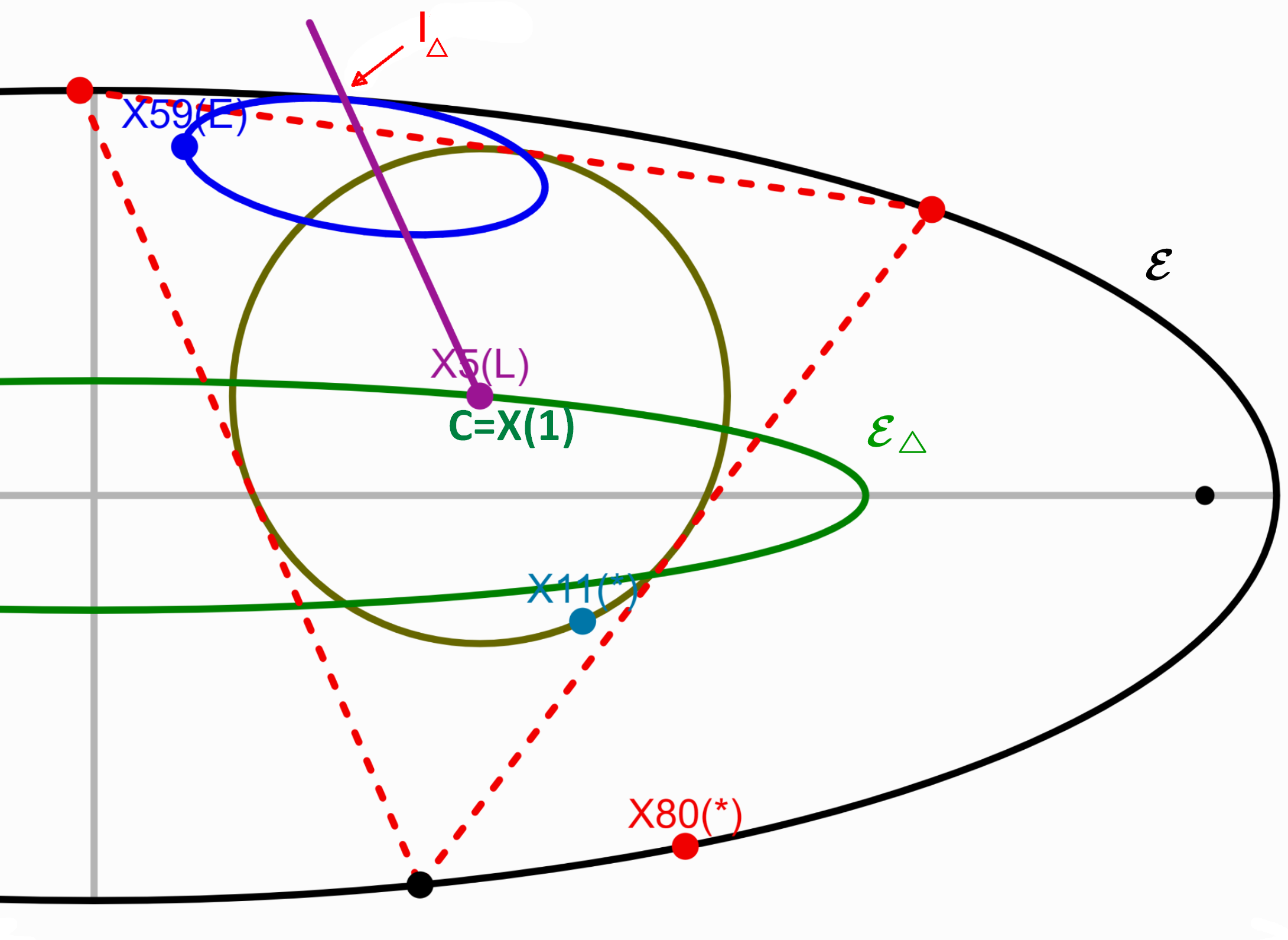}
\caption{When the incenter is on $\Et$ (dark green), $\L_5$ is a line, and Feuerbach's point $X_{11}$ and $X_{80}$ are stationary on the incircle and $\E$, respectively. $\L_{59}$ is an ellipse (blue), internally-tangent to $\E$ at $\It$ (see \cref{lem:isosceles}) In the picture, the Euler center $X_5$ is on the incenter $X_1$, i.e., the Poncelet triangle is instantaneously equilateral. When this happens, $X_{59}$ is at a vertex of its locus. Live: \hrefs{https://bit.ly/44yMrz8}}
\label{fig:x59-cut}
\end{figure}


\section*{Acknowledgements}
\noindent Comments by two outstanding referees were crucial in improving this article. We are also indebted to Arseniy Akopyan, Clark Kimberling, Peter Moses, and Richard Schwartz for valuable help. The first author is a fellow of CNPq.

\appendix

\section{Affine triples}
\label{app:affine-triples}
\cref{tab:triples} lists the $\alpha,\beta,\gamma$ used to express a given triangle center $X_k$ as the linear combination $\alpha X_1+\beta X_2+\gamma X_3$, where $\rho=r/R$ (a more complete table appears in \cite[Table 1]{helman2021-theory}):
\begin{table}[H]
\begin{tabular}{|c|c|c|c|c|c|c|c|c|c|c|c|c|}
\hline
& $X_1$ & $X_2$ & $X_3$ & $X_4$ & $X_5$ & $X_7$ & $X_8$ & $X_{11}$ & $X_{12}$ & $X_{20}$ & $X_{36}$ & $X_{80}$\\
\hline
$\alpha$ & $1$ & $0$ & $0$ & $0$ & $0$ & $\frac{2\rho+4}{\rho+4}$ & $-2$ & $\frac{1}{1-2\rho}$ & $\frac{1}{1+2\rho}$ & $0$ & $\frac{1}{1-2\rho}$ & $\frac{2\rho+1}{1-2\rho}$ \\
\hline
$\beta$ & $0$ & $1$ & $0$ & $3$ & $\frac{3}{2}$ & $\frac{3\rho}{\rho+4}$ & $3$ & $\frac{-3\rho}{1-2\rho}$ & $\frac{3\rho}{1+2\rho}$ & $-3$ & $0$ & $\frac{-6\rho}{1-2\rho}$ \\
\hline
$\gamma$ & $0$ & $0$ & $1$ & $-2$ & $-\frac{1}{2}$ & $\frac{-4\rho}{\rho+4}$ & $0$ & $\frac{\rho}{1-2\rho}$ & $\frac{-\rho}{1+2\rho}$ & $4$ & $\frac{-2\rho}{1-2\rho}$ & $\frac{2\rho}{1-2\rho}$ \\
\hline
\end{tabular}
\caption{Affine triples for triangle centers used in the text.}
\label{tab:triples}
\end{table}


\clearpage
\section{Triangle centers}
\label{app:centers}
Definitions of triangle centers mentioned in the text appear in \cref{tab:centers}.

\begin{table}[H]
\small
\begin{tabular}{|@{\hspace{2pt}}c@{\hspace{2pt}}|@{\hspace{2pt}}c@{\hspace{2pt}}|@{\hspace{2pt}}l@{\hspace{2pt}}|@{\hspace{2pt}}l@{\hspace{2pt}}|p{8cm}|}
\hline
center & name & first barycentric & construction \\
\hline
$X_1$ & incenter & $l_1::$ &  meet of angle bisectors \\
$X_2$ & barycenter/centroid & $1:: $ & meet of medians  \\
$X_3$ & circumcenter & $l_1^2(l_2^2 + l_3^2 - l_1^2)::$ & meet of perpendicular bisectors \\
$X_4$ & orthocenter & $(l_2^2+l_3^2-l_1^2)^{-1}::$ &  meet of altitudes \\
$X_5$ & Euler center & $l_1^2(l_2^2 +l_3^2)-(l_2^2-l_3^2)^2::$ & center of the Euler circle \\
$X_7$ & Gergonne point & $(l_2+l_3-l_1)^{-1}::$ & contact triangle perspector \\
$X_8$ & Nagel's point & $l_2+l_3-l_1::$ & extouch triangle perspector \\
$X_{11}$ & Feuerbach's point & $(l_2+l_3-l_1)(l_2-l_3)^2::$ & incircle and Euler circle touchpoint \\
$X_{12}$ & \makecell[ct]{in-similitude center \\ (incircle, Euler circle)} & $(l_2 + l_3)^2 (l_2 + l_3 - l_1)^{-1}::$ & \\
$X_{20}$ & de Longchamps' point & $-3 l_1^4 + 2 l_1^2(l_2^2 + l_3^2) + (l_2^2 - l_3^2)^2::$ & reflection of $X_4$ about $X_3$ \\
$X_{36}$ & circumcircle-inverse of $X_1$  & $l_1^2(l_2^2+l_3^2-l_1^2- l_2 l_3)::$ & isogonal conjugate of $X_{80}$  \\
$X_{59}$ & isogonal conjugate of $X_{11}$ & $l_1^2 (l_2+l_3-l_1)^{-1} (l_2-l_3)^{-2}::$ &  \\
$X_{80}$ & reflection of $X_1$ on $X_{11}$ & $(l_2^2+l_3^2-l_1^2-l_2 l_3)^{-1}::$ & isogonal conjugate of $X_{36}$ \\
$X_{106}$ & \makecell[ct]{isog. conj. of the inters.\\ of $X_1 X_2$ w/ line at infinity.} &  $l_1^2 (2 l_1 - l_2 - l_3)^{-1}$ :: & called $\Lambda(X_1,X_2)$ on \cite{etc} \\
\hline
\end{tabular}
\caption{Kimberling codes for various triangle centers mentioned here, along with their names, barycentric coordinates (only the first shown, the other two can be obtaine by cyclical replacement), and construction notes \cite{etc}. The isogonal conjugate of a point with barycentrics $[z_1:z_2:z_3]$ is $[l_1^2/z_1 : l_2^2/z_2 : l_3^2/z_3]$, where the $l_i$ are the sidelengths.}
\label{tab:centers}
\end{table}

\section{Table of Symbols}
\label{app:symbols}
Symbols used above appear in \cref{tab:symbols}.

\begin{table}[H]
\small
\begin{tabular}{|c|p{8cm}|}
\hline
symbol & meaning \\
\hline
$\mathbb{T},\mathbb{D}$ & unit circle (resp. disk) in complex plane \\ 
$f,g$ & foci of caustic of affine image of $\T$ for which $\E$ is a circle \\
$\lambda$ & parameter of symmetric parametrization \\
\hline
$\T$ & Poncelet triangle family interscribed between two generic conics \\
$\To$ & Poncelet triangle family about the incircle \\
$\Tt$ & the family $\To$ containing an equilateral \\
\hline
$\E, \E_c$ & Poncelet conics (incidence and tangency)  \\
$O, O_c$ & their centers \\
$a,b,c$ & semi-axis' and half-focal lengths of $\E$ \\
$a_c,b_c,c_c$ & idem for $\E_c$ \\
$\K, C=[x_c,y_c]$ & incircle and its center (= $X_1$) \\
$T, l_i, \theta_i$ & a Poncelet triangle, its sidelengths and internal angles \\
$r, R$ & inradius and circumradius of a $T$ \\
$\Et, a_{\triangle}, b_{\triangle}$ & (elliptic) locus of $\E$-inscr.\@~equilat.\@~centroids and semi-axes \\
\hline
$X_k,\L_k$ & triangle center and its locus over Poncelet, typically $\To$\\
$C_i=(x_i,y_i), a_i, b_i$ & center and semiaxes of (the conic) $\L_i$ \\
$F_i,F_i'$ & foci of the (conic) locus of $X_i$ \\
$X_k(z), z\in\{*,C,E,L\}$ & in figures, indicates if the $X_k$ locus is a point, circle, ellipse, or line, resp. \\
\hline
\end{tabular}
\caption{Symbols used in the article.}
\label{tab:symbols}
\end{table}

\bibliographystyle{maa-eprint}
\bibliography{refs,refs_00_book,refs_01_pub,refs_03_sub,refs_04_unsub}

\end{document}